\documentclass[reqno]{amsart}
\usepackage[active]{srcltx}
\usepackage{array}
\usepackage{amsmath,amscd,amssymb}
\usepackage{latexsym}
\usepackage{epsfig}
\usepackage[all]{xy}
\usepackage[T1]{fontenc}
\usepackage{color}

\newtheorem{theorem}{Theorem}[section] 
\newtheorem{lemma}[theorem]{Lemma}
\newtheorem{proposition}[theorem]{Proposition}
\newtheorem{corollary}[theorem]{Corollary}

\newtheorem{definition}[theorem]{Definition}

\newcommand{\CC}{\mathbb C}

\newcommand{\HH}{\mathbb H}
\newcommand{\NN}{\mathbb N}
\newcommand{\PP}{\mathbb P}
\newcommand{\QQ}{\mathbb Q}
\newcommand{\RR}{\mathbb R}
\newcommand{\ZZ}{\mathbb Z}
\newcommand{\cD}{\mathcal D}

\newcommand{\cH}{\mathcal H}
\newcommand{\cX}{\mathcal X}
\newcommand{\cY}{\mathcal Y}
\newcommand{\cZ}{\mathcal Z}
\newcommand{\ffZ}{\mathfrak Z}

\newcommand{\Sp}{\mathop{\mathrm {Sp}}\nolimits}

\newcommand{\SL}{\mathop{\mathrm {SL}}\nolimits}
\newcommand{\SO}{\mathop{\mathrm {SO}}\nolimits}
\newcommand{\Orth}{\mathop{\null\mathrm {O}}\nolimits}

\newcommand{\Ker}{\mathop{\mathrm {Ker}}\nolimits}

\newcommand{\Lift}{\mathop{\mathrm {Lift}}\nolimits}

\newcommand{\diag}{\mathop{\mathrm {diag}}\nolimits}

\renewcommand{\Im}{\mathop{\mathrm {Im}}\nolimits}

\newcommand{\rank}{\mathop{\mathrm {rank}}\nolimits}

\newcommand{\id}{\mathop{\mathrm {id}}\nolimits}

\newcommand{\Eins}{{\mathbf 1}}

\newcommand{\latt}[1]{{\langle{#1}\rangle}}
\renewcommand{\div}{\mathop{\mathrm {div}}\nolimits}
\newcommand{\Kthree}{\mathop{\mathrm {K3}}\nolimits}

\begin{document}
\title{Modular forms of orthogonal type and Jacobi theta-series}
\author{F.~Cl\'ery and V.~Gritsenko}
\date{}
\begin{abstract}
In this paper we consider Jacobi forms of half-integral index
for any positive definite lattice $L$ (classical Jacobi forms
from the book of Eichler and Zagier correspond to the lattice
$A_1=\latt{2}$). We give a lot of examples of Jacobi forms of singular
and critical weights for root systems using Jacobi theta-series.
We give the Jacobi lifting for  Jacobi forms of half-integral indices.
In some case it gives additive lifting construction of new reflective
modular forms.
\end{abstract}
\maketitle
\section{Introduction}
\thispagestyle{empty}

The divisor of a reflective modular form with respect to an integral  orthogonal group of signature $(2,n)$ is determined by
reflections. Such modular forms determine Lorentzian
Kac--Moody Lie (super) algebras. The most famous reflective
modular form is the Borcherds function $\Phi_{12}$ with respect
to $\Orth^+(II_{2,26})$ which determines the Fake Monster
Lie algebra (see \cite{B1}). One can consider reflective modular
forms as  automorphic discriminants or
multi-dimensional Dedekind $\eta$-functions (see \cite{B2}--\cite{B3},
\cite{GN1}--\cite{GN4}). Reflective modular forms  play
also an important role in complex algebraic geometry
(see \cite{GHS} and \cite{G10}).
All of them are Borcherds automorphic products  and some of them
can be constructed as additive (or Jacobi) lifting.
If a reflective modular form
can be obtained by additive (Jacobi) lifting then
one has a simple formula for its Fourier coefficients
which determine the generators and relations of Lorentzian
Kac--Moody algebras (see \cite{GN1}).

In \cite{G10} the second author constructed
the Borcherds-Enriques form $\Phi_4$, the automorphic discriminant
of the moduli space of Enriques surfaces (see \cite{B2}),
as Jacobi lifting,
$\Lift\bigl(\vartheta(\tau,z_1)\dots \vartheta(\tau,z_8)\bigr)$,
 of the tensor product of eight classical  Jacobi theta-series
(see \cite{G-K3} for the definition of Lift which provides
a modular form on orthogonal group by its first Fourier--Jacobi coefficient).
This new construction of $\Phi_4$ gives an answer to a problem
formulated  by
K.-I. Yoshikawa (\cite{Y}) and to a question of Harwey
and Moore (\cite{HM}) about the second Lorentzian  Kac--Moody super
Lie algebras determined by the Borcherds--Enriques form $\Phi_4$
and its quasi-pullbacks.

Another application of reflective modular forms of type
$\Lift(\vartheta(\tau,z_1)\dots \vartheta(\tau,z_8))$
is the construction of new examples of modular varieties of orthogonal type
of Kodaira dimension $0$
(see the beginning of \S 2).
The first two examples of this type of dimension $3$
are related to reflective Siegel cusp forms of weight $3$ and
Siegel modular three-folds having compact Calabi--Yau models
(see \cite{GH1}, \cite{CG} and \cite{FS-M}).
In the case of dimension $4$ the unique cusp form of weight $4$
was defined in \cite{G10} as a Borcherds product
but it can also be  constructed as a lifting of a Jacobi form
of half-integral index with a character of order $2$
of the full modular group $\SL_2(\ZZ)$ (see Example 2.4 in \S 2).
Jacobi forms of half-integral index in one variable are very important in the theory of Lorentzian Kac--Moody algebras
of hyperbolic rank $3$  corresponding to  Siegel modular forms
(see \cite{GN1}--\cite{GN3}). Moreover they are very natural
in the structure theory of classical Jacobi forms (in the sense
of Eichler and Zagier \cite{EZ}) and
in applications to topology and string theory
(see \cite{DMVV},  \cite{G-EG}).

In this paper we consider Jacobi forms of half-integral index
for any positive definite lattice $L$ (Jacobi forms
in \cite{EZ} correspond to the lattice $A_1=\latt{2}$).
Jacobi forms in many variables naturally appeared
in the theory of affine Lie algebras (see \cite{K} and \cite{KP}).
One can consider Jacobi forms  as vector valued modular forms
in one variable. Vector valued modular forms are used in
the additive Borcherds lifting (see \cite[\S 14]{B3})
which is a genralization of the Jacobi lifting of  \cite{G-K3}.
In this paper we follow the  general approach to Jacobi forms proposed
in \cite{G-K3}.
The first section contains all necessary definitions and basic results
on Jacobi forms in many variables.
It turns out that the order of the character of the integral Heisenberg
group of such Jacobi forms is always at most  $2$
(see Proposition \ref{pr-index}). Using the classical Jacobi
theta-series we give examples of Jacobi forms for
the root lattices.  We show at the  end of the first section
(Examples 1.8--1.11) that the natural theta-products
give all Jacobi forms of singular weight
(or vector valued $\SL_2(\ZZ)$-modular forms of weight $0$
related to the Weil representation)
for the lattices $D_m$.

In \S 2 we give the Jacobi lifting for  Jacobi
forms of half-integral index with a possible character.
This explicit construction has many advantages:
one can see immediately a part of its divisor,
the maximal modular group of the lifting, etc. We  construct
many modular forms of singular, critical and canonical weights
on orthogonal groups. In particular, we give in Example 2.4
natural reflective generalizations of the classical Igusa modular
form $\Delta_5$.

In \S 3 we analyze  Jacobi forms of singular (the minimal possible)
and critical (singular weight$+\frac{1}2$) weights using the theta-products
and their pullbacks (see Proposition 3.1--3.2).
 In this way we construct many examples
(see Propositions \ref{Ptheta-Am}--\ref{Form-crit}).
In particular, using this approach  we give a new explanation
of theta-quarks  in Corollary \ref{singJ-A2}, which are
the simplest examples of holomorphic  theta-blocks
(see \cite{GSZ}).
\smallskip

{\bf Acknowledgements:}
This work was supported by the grant ANR-09-BLAN-0104-01.
The authors are grateful to the Max-Planck-Institut
f\"ur  Mathematik in Bonn for support and for providing excellent working conditions.

\section{The Jacobi group and Jacobi modular forms}

In this section we discuss Jacobi forms of orthogonal type.
In the definitions below we follow the paper \cite{G1}--\cite{G-K3}
where Jacobi forms were considered as modular forms with respect to
a parabolic subgroup of an orthogonal group of signature $(2,n)$.

By a lattice we mean a free $\ZZ$-module equipped  with a
non-degenerate symmetric bilinear form $(\cdot , \cdot)$ with values
in $\ZZ$. A lattice is even if $(l,l)$ is even for all its elements.
Let $L_2$ be a lattice of signature $(2,n_0+2)$. All the bilinear forms
we deal with can be extended to $L_2\otimes \CC$ (respectively to
$L_2\otimes \RR$) by $\CC$-linearity (respectevely by $\RR$-linearity) and we
use the same notations for these extensions. Let
$$
\cD(L_2)=\{[\cZ] \in \PP(L_2\otimes \CC) \mid (\cZ,\cZ)=0,
\ (\cZ,\overline{\cZ})>0\}^+
$$
be the  $(n_0+2)$-dimensional classical Hermitian domain 
of type IV associated to the lattice $L_2$
(here $+$ denotes one of its two connected components).
We denote by $\Orth^+(L_2\otimes \RR)$ the index $2$ subgroup of the real  orthogonal group
preserving $\cD(L_2)$. Then $\Orth^+(L_2)$ is  the intersection of the integral
orthogonal group $\Orth(L_2)$ with $\Orth^+(L_2\otimes \RR)$. We use the  similar notation
$\SO^+(L_2)$ for the special orthogonal group.

In this paper we assume that $L_2$ is  an even lattice of signature
$(2,n_0+2)$ containing two hyperbolic planes
$$
L_2=U\oplus U_1\oplus L(-1), \qquad U_1=U\simeq
\begin{pmatrix}
0&1\\1&0
\end{pmatrix}
$$
where $L(-1)$ is  a negative definite even lattice of rank $n_0$.
The restriction of the bilinear form $(. ,.)$ to $L(-1)$
i.e. $(. , .)|L(-1)$ is denoted by $-( .,. )$ with $(. , .)$ definite positive.
We fix a basis of the hyperbolic plane $U=\ZZ e\oplus\ZZ f$:
$e\cdot f=(e,f)=0$ and $e^2=f^2=0$. Similarly  $U_1=\ZZ e_1\oplus\ZZ f_1$.
Let $F$ be the totally isotropic plane spanned by $f$ and $f_1$ and let $P_F$
be the parabolic subgroup of $\SO^+(L_2)$ that preserves~$F$. This
corresponds to a $1$-dimensional cusp of the modular variety
$\SO^+(L_2)\backslash\cD(L_2)$ (see \cite{GHS}).
We choose a basis of $L_2$ of the form
$(e,e_1,\dots , f_1,f)$ where $\dots$ denote a basis of $L(-1)$.
In this basis the quadratic form associated to the bilinear form on $L_2$
has the following Gram's matrix
$$
S_2=\left(\begin{array}{cc|ccc|cc}
0&0&0&\cdots&0&0&1\\
0&0&0&\cdots&0&1&0\\
\hline
0&0& & & &0&0\\
\vdots&\vdots& &-S& &\vdots& \vdots\\
0&0& & & &0&0\\
\hline
0&1&0&\cdots&0&0&0\\
1&0&0&\cdots&0&0&0\\
\end{array}\right)
$$
where $S$ is a positive definite integral matrix with even entries
on the main diagonal.
In this paper we denote the {\it positive definite} even integral bilinear form
on the lattice $L$ by $(.,.)$ and the bilinear form of signature $(1,n_0+1)$
on the hyperbolic lattice $U_1\oplus L(-1)$ by $(.,.)_1$. Therefore for any
$v=ne_1+l+mf_1\in L_1$
we have $(v,v)_1=2nm-(l,l)$.

The subgroup $\Gamma^J(L)$ of $P_F$ of
elements acting trivially on the sublattice $L$ is called the \emph{Jacobi group}.
The Jacobi group has a subgroup isomorphic to $\SL_2(\ZZ)$.
For any $A=\left(\smallmatrix a&b\\c&d \endsmallmatrix\right)\in \SL_2(\ZZ)$
we put
\begin{equation}\label{A-im}
\{A\}:=\begin{pmatrix}
A^*&0&0\\
0&\Eins_{n_0}&0\\
0&0&A
\end{pmatrix}\in \Gamma^J(L)\quad\text{where\ }\
A^*=\left(\smallmatrix 0&1\\1&0\endsmallmatrix\right) {}^tA^{-1}
\left(\smallmatrix 0&1\\1&0\endsmallmatrix\right)=
\left(\smallmatrix \ a&-b\\-c&\ d \endsmallmatrix\right).
\end{equation}
The second standard subgroup of $\Gamma^J(L)$
is the Heisenberg group $H(L)$ acting trivially on the totally isotropic plane $F$.
This is  the central extension $\ZZ\rtimes (L\times L)$.
More precisely we define
$$
H(L)=\{[x,y;r]\,:\ x, y\in L,\
r\in \frac{1}{2}\,\ZZ\ \text{\ with }\ r+\frac{1}{2}(x,y)\in \ZZ\}
$$
where
\begin{equation}\label{H(L)}
[x,y;r]:=
\begin{pmatrix}
1&0&{}^tyS& (x,y)/2-r& (y,y)/2\\
0&1&{}^txS&(x,x)/2& (x,y)/2+r\\
0&0&\Eins_{n_0}&x&y\\
0&0&0&1&0\\
0&0&0&0&1\end{pmatrix}
\end{equation}
with $S$ the positive definite Gram matrix of the quadratic form $L$ of rank
$n_0$, $(x,y)={}^txSy$ and  we consider $x$ and $y$ as column vectors.
The multiplication in $H(L)$ is given by the following
formula
\begin{equation}\label{eq-H-prod}
[x,y;r]\cdot [x',y';r']=[x+x',y+y';r+r'+\frac{1}{2}((x,y')-(x',y))].
\end{equation}
In particular, the center of $H(L)$ is equal to $\{[0,0;r],\ r\in \ZZ\}$.
We introduce a subgroup of $H(L)$
$$
H_s(L)=<[x,0;0],\  [0,y;0]\,|\, x,y\in L>
$$
with  a smaller center and we  call  it {\it minimal} integral Heisenberg group of $L$.
The group $\SL_2(\ZZ)$ acts on $H(L)$ by conjugation:
\begin{equation}\label{eq-A-action}
A.[x,y;r]:=\{A\}[x,y;r]\{A^{-1}\}=
[dx-cy,-bx+ay;r].
\end{equation}
Using  \eqref{eq-H-prod} and \eqref{eq-A-action} one can define the integral Jacobi
group or its subgroup $\Gamma_s^J(L)$ as the semidirect product
of $\SL_2(\ZZ)$ with the Heisenberg group $H(L)$ or $H_s(L)$
$$
\Gamma^J(L)\simeq \SL_2(\ZZ)\rtimes H(L)\quad
{\rm and}\quad
\Gamma_s^J(L)\simeq \SL_2(\ZZ)\rtimes H_s(L)
$$
Extending the coefficients we can define the real Jacobi group which is a subgroup
of the real orthogonal group
$\Gamma^J(L\otimes \RR)\simeq \SL_2(\RR)\rtimes H(L\otimes \RR)$.

In what follows we need characters of Jacobi groups. Let $\chi:
\Gamma^J(L)\to \CC^*$ be a character of finite order. Its
restriction to $\SL_2(\ZZ)$, $\chi|_{\SL_2(\ZZ)}$, is an even power
$v_\eta^{D}$ of the multiplier system  of the Dedekind
$\eta$-function and we have
\begin{equation}\label{J-chi}
\chi(\{A\}\cdot [x,y;r])=
v_\eta^D(A)\cdot \nu([x,y;r]), \quad\text{ where }\
\chi|_{\SL_2(\ZZ)}=v_\eta^{D}, \ \ \nu=\chi|_{H(L)}.
\end{equation}
If $D$ is odd then we obtain a multiplier system of the Jacobi group.
The properties of the character of the Heisenberg group are clarified
by the next proposition.

\begin{proposition}\label{pr-char}
{\bf 1.} Let $s(L)\in \NN^*$ (resp. $n(L)$) denote the generator of the integral ideal
generated by $(x,y)$  (resp. $(x,x)$) for $x$ and $y$ in $L$.
Let $\left[H(L),H(L)\right]$ be the derivated group of $H(L)$.
Then
$$
\left[H(L),H(L)\right]=\left[H_s(L),H_s(L)\right]=\{[0,0; r]\,|\, r\in  s(L)\ZZ\}.
$$
This subgroup is the center of $H_s(L)$.

{\bf 2.}
Let $\nu: H(L)\to \CC^*$ be a  character  of finite order which is invariant
with respect to $\SL_2(\ZZ)$-action $\nu(A.[x,y;r])=\nu([x,y;r])$. Then
$$
\nu([x,y;r])=e^{\pi i t ((x,x)+(y,y)-(x,y)+2r)}
$$
where $t\in \QQ$ such that $t\cdot s(L)\in \ZZ$.
The restriction $\nu|_{H_s(L)}$ is a binary character which  is trivial if
$t\cdot n(L)\in 2\ZZ$.
\end{proposition}

{\bf Remark.} The constants $s(L)$ and $n(L)$ are called
{\it scale} and {\it norm} of the integral lattice $L$.
The scale $s(L)$ is  the greatest common divisor of the entries of the Cartan matrix $S$
of $L$. For any even lattice $L$ the norm $n(L)$ is an even divisor of $s(L)$.

\begin{proof}
The first property follows from the formula for the commutator of the
elements of $H(L)$
\begin{equation}\label{eq-H-comt}
[x,y;r]\cdot [x',y';r']\cdot [x,y;r]^{-1}\cdot [x',y';r']^{-1}=
[0,0;(x,y')-(x',y)]
\end{equation}
because $[x,y;r]^{-1}=[-x,-y;-r]$.

Considering the restriction of the character to the center of $H(L)$, isomorphic
to $\ZZ$, we get
$\nu([0,0;r])=\exp(2\pi i t r)$ with  $t\in \QQ$.
The formula for the commutator \eqref{eq-H-comt} shows that $t\cdot s(L)\in \ZZ$.

The invariance of the character with respect to
$A=-\Eins_{2}$,
$A=\left(\begin{smallmatrix} 0&-1\\1&\ 0\end{smallmatrix}\right)$
and
$A=\left(\begin{smallmatrix} 1&1\\0&1\end{smallmatrix}\right)$ gives us
$
\nu([x,y;r])=\nu([-x,-y;r])=\nu([-y,x;r])=\nu([x,y-x;r])
$.
Therefore $\nu([x,0;0])=\nu([-x,0;0])=\nu([x,0;0])^{-1}$,
$\nu([0,y;0])=\nu([y,0;0])$ and $\nu([x,0;0])=\nu([x,-x;0])$.
We have
$$
[x,-x;0]=[x, 0;0]\cdot [0, -x; 0]\cdot [0,0,\tfrac{1}2(x,x)].
$$
Therefore $\nu([x,0;0])=\nu([x,-x;0])=e^{\pi i t (x,x)}$ and the final formula
follows from the decomposition
$$
[x,y;r]=[x,y;\tfrac{1}2(x,y)]\cdot [0,0;r-\tfrac{1}2(x,y)]=
[x,0;0]\cdot[0,y;0]\cdot [0,0;r-\tfrac{1}2(x,y)].
$$
We see that $t\cdot (x,x)\in \ZZ$. Therefore the order of $\nu|_{H_s(L)}$
is equal to $1$ or $2$.
\end{proof}
In order to define Jacobi forms we have to fix a tube realization
of the homogeneous domain $\cD(L_2)$ related to the $1$-dimensional
boundary component of its Baily--Borel compactification corresponding
to the Jacobi group related to the isotropic tower $\latt{f}\subset \latt{f,f_1}$.
Let $[\cZ]=[\cX+i\cY] \in \cD(L_2)$. Then
$$
(\cX,\cY)=0,\ (\cX,\cX)=(\cY,\cY)\ \text{\ and \ }(\cZ,\overline{\cZ})=2(\cY,\cY)>0.
$$
Using the basis $\latt{e,f}_\ZZ=U$ we write $\cZ=z'e+\widetilde Z+zf$ with
$\widetilde Z\in L_1\otimes \CC$, where
$$
L_1=U_1\oplus L(-1)
$$
is the hyperbolic lattice of signature $(1,n_0+1)$ with the bilinear form $(\cdot,\cdot)_1$.
We note that $z\ne 0$. (If $z=0$ then  the real and imaginary parts of
$\widetilde Z$
form two orthogonal vectors of positive norm in the hyperbolic lattice $L_1\otimes \RR$.)
Thus $[{}^t\cZ]=[(-\frac{1}2(Z,Z)_1, {}^tZ,1)]$. Using the basis $\latt{e_1,f_1}_\ZZ=U_1$
of the second hyperbolic plane in $L$ we see that
$\cD(L_2)$ is isomorphic to the tube domain
$$
\cH(L)=\cH_{n_0+2}(L)=
\{Z=\left(\smallmatrix\omega\\ \mathfrak{Z}\\ \tau
\endsmallmatrix\right),\
\tau,\,\omega\in\HH_1,\  \ffZ\in L\otimes \CC,\
(\Im Z, \Im Z)_1>0\}
$$
where
$$
(\Im Z, \Im Z)_1=2\Im(\omega)\Im(\tau)-(\Im(\ffZ),\Im(\ffZ))>0.
$$
We fix the isomorphism $[\rm{pr}]:\cH(L)\to\cD(L_2)$
defined
by the  $1$-dimensional  cusp $F$ fixed above
\begin{equation}\label{pr-Z}
Z=\left(\begin{smallmatrix}\omega\\ \ffZ \\ \tau\end{smallmatrix}\right)\
{\mapsto}\ {{\rm{pr}}(Z)}=
\left(\begin{smallmatrix}-\frac{1}{2}(Z,Z)_1\\\omega\\ \ffZ \\
\tau\\1\end{smallmatrix}\right)
\mapsto\ \left[{\rm{pr}}(Z)\right].
\end{equation}
The map $\rm{pr}$  gives us the embedding of
$\mathcal{H}(L)$ into the affine cone $\cD^\bullet(L_2)$ over
$\cD(L_2)$. Using the map $[\rm{pr}]$, we can define a
linear-fractional action of $M\in \Orth^+(L_2\otimes \RR)$ on the tube
domain
$$
M\cdot {\rm{pr}}(Z)=J(M,Z)\cdot {\rm{pr}}(M\latt{Z})
$$
where the automorphic factor $J(M,Z)$ is the last (non-zero) element
of the column vector $M\cdot {\rm{pr}}(Z)\in \cD^\bullet(L_2)$.
In particular for the standard elements of  the Jacobi group we have the following
action
$$
\{A\}\langle Z \rangle=
{}^t\bigl(\omega-\frac{c(\mathfrak{Z},\mathfrak{Z})}{2(c\tau+d)},\,
\frac{{}^t\mathfrak{Z}}{c\tau+d},\, \frac{a\tau+b}{c\tau+d}),
\quad
A=\begin{pmatrix}a&b\\c&d\end{pmatrix}\in \SL_2(\RR);
$$
$$
\left[x,y;r\right]\langle Z \rangle=
{}^t\bigl(\omega+\tfrac{1}{2}(x,x)\tau+(x,\mathfrak{Z})+\tfrac{1}{2}(x,y)+r,
\,{}^t(\mathfrak{Z}+x\tau+y),\, \tau\bigl)
$$
where $x,y\in L\otimes \RR$, $r\in \RR$.
We note that
$J(\{A\},Z)=c\tau+d$ and $J([x,y;r],Z)=1$.

For a function $\psi: \cH(L)\to \CC$, we define as usual
$$
(\psi\vert_{k}M)(Z):= J(M,Z)^{-k}\psi(M\latt{Z}),
\qquad M\in \Orth^+(L_2\otimes \RR).
$$
In the next definition, Jacobi forms are considered as modular forms with
respect to the parabolic subgroup $\Gamma^J(L)$ of $\Orth^+(L_2)$.
\begin{definition}\label{def-JF}
Let $\chi$ be a character (or a multiplier system) of finite order of $\Gamma^J(L)$,
$k$ be  integral or half-integral and
$t$ be a (positive) rational number.
A holomorphic function  $\varphi: \HH_1\times (L\otimes \CC)\to \CC$
is called a holomorphic
{\bf{Jacobi form}} of weight $k$ and index $t$ with a character (or a multiplier system) $\chi$
if the function
$$
\widetilde{\varphi}(Z)=\varphi(\tau,\mathfrak{Z})e^{2i\pi t\omega},
\quad Z=\left(\begin{smallmatrix}\omega \\ \mathfrak{Z} \\ \tau
\end{smallmatrix}\right) \in \cH(L)
$$
satisfies the functional equation
\begin{equation}\label{eq-JF}
\widetilde{\varphi}\vert_{k}M=\chi(M)\widetilde{\varphi} \quad
{\rm for\  any\ } M\in\Gamma^J(L)
\end{equation}
and  is holomorphic at ``infinity'' (see the condition
\eqref{F-exp-phi} below).
\end{definition}

\noindent
{\bf Remarks.} 1) We show  below that for any non zero Jacobi form of rational index $t$
we have $t\cdot s(L)\in \ZZ$  where  $s(L)$ is  the scale of the lattice
(see Proposition \ref{pr-char}).

2) One can reduce this definition to the only two cases $t=1$ and $t=\frac{1}2$
(see Proposition \ref{pr-index}).

3) One can give another  definition of Jacobi forms in more intrinsic terms
(see Definition 1.2').
\smallskip

In order to precise  the condition ``to be holomorphic at infinity''
we analyse the character $\chi$, the functional equation  and the
Fourier expansion of Jacobi forms. We  decompose the character into
two parts
$$
\chi=\chi|_{\SL_2(\ZZ)}\times \chi|_{H(L)}=\chi_1\times \nu,\qquad \text{where }
\chi_1=v_\eta^D
$$
(see \eqref{J-chi})
and $\nu$ satisfies the condition of Proposition \ref{pr-char}.
For a central element $[0,0;(x,y)]\in H_s(L)$ the equation \eqref{eq-JF} gives
$\nu([0,0;(x,y)])=e^{2 \pi i t (x,y)}=1$. Therefore
$$
t\cdot s(L)\in \ZZ\quad {\text{if $\varphi$ is not identically zero}}
$$
and
$$
\nu([x,y;r])=e^{\pi i t ((x,x)+(y,y)-(x,y)+2r)},
\qquad [x,y;r]\in H(L)
$$
as in Proposition \ref{pr-char}.

The formulae above for the action of the generators of the Jacobi group
on the tube domain  define also  an action, denoted by $M\langle \tau,\mathfrak{Z}\rangle$,
of the real Jacobi group  $\Gamma^J(L\otimes \RR)$
on the domain $\HH_1\times (L\otimes \CC)$.
We can always add to any
$(\tau,\mathfrak{Z})\in \HH_1\times (L\otimes \CC)$
a complex number $\omega\in \HH_1$ such that
$Z=\left(\begin{smallmatrix}\omega\\ \mathfrak{Z}\\ \tau\end{smallmatrix}\right)$
belongs to $\cH(L)$.
If we denote the first component of $M\langle Z\rangle$
(that is the component along the vector $e_1$ of our basis)
by $\omega\{M\langle Z\rangle\}$ for $M\in \Gamma^J(L\otimes \RR)$ then
$$
J_{k,t}(M;\tau,\mathfrak{Z})=J(M,Z)^ke^{-2i\pi t\omega\{M\langle Z\rangle\}}e^{2i\pi t\omega}
$$
defines an automorphic factor of weight $k$ and index $t$ for the Jacobi group.
For the generators of the Jacobi group, we get
$$
J_{k,t}(\{A\};\tau,\mathfrak{Z})=(c\tau+d)^k e^{i\pi
t\frac{c(\mathfrak{Z},\mathfrak{Z})}{c\tau+d}}, \quad
A=\left(\begin{smallmatrix}a&b\\c&d\end{smallmatrix}\right)\in
\SL_2(\RR)
$$
and
$$
J_{k,t}(\left[x,y;r\right];\tau,\mathfrak{Z})=
e^{-2i\pi t(\frac{1}{2}(x,x)\tau+(x,\mathfrak{Z})+\frac{1}{2}(x,y)+r)}.
$$
We also get an action of the Jacobi group on the space of functions
defined on $\HH_1\times (L\otimes \CC)$:
$$
(\varphi\vert_{k,t}M)(\tau,\mathfrak{Z}):=
J^{-1}_{k,t}(M;\tau,\mathfrak{Z})\varphi(M\langle \tau,\mathfrak{Z}\rangle).
$$
Then the equation $\eqref{eq-JF}$ in the definition of Jacobi forms  is equivalent to
$$
(\varphi\vert_{k,t}M)(\tau,\mathfrak{Z})=\chi(M)\varphi(\tau,\ffZ),
\qquad M\in\Gamma^J(L).
$$
For the generators of the  Jacobi group we obtain
\begin{equation}\label{jacobi-A}
\varphi(\frac{a\tau+b}{c\tau+d},\frac{\mathfrak{Z}}{c\tau+d})
=\chi(A)(c\tau+d)^k e^{i\pi t\, \frac{c(\mathfrak{Z},\mathfrak{Z})}{(c\tau+d)}}
\varphi(\tau,\mathfrak{Z})
\end{equation}
for all
$A=\left(\begin{smallmatrix}a&b\\c&d\end{smallmatrix}\right)\in\SL_2(\ZZ)$
and
\begin{equation}\label{jacobi-H}
\varphi(\tau,\mathfrak{Z}+x\tau+y)=\chi([x,y; \tfrac{1}{2}(x,y)])
e^{-i\pi t ((x,x)\tau+2(x,\mathfrak{Z}))}\varphi(\tau,\mathfrak{Z})
\end{equation}
for all $x,y\in L$.

\noindent (We note that $t(x,y)\in \ZZ$ for any $x$, $y$ in $L$ if
$\varphi\not\equiv 0$.) The variables $\tau$ and $\ffZ$ are called
modular and abelian variables. To clarify the last condition of
Definition \ref{def-JF} we consider the Fourier expansion of a
Jacobi form $\varphi$.

We see that the function $\varphi$ have the following periodic
properties
$$
\varphi(\tau+1,\ffZ)=e^{2\pi i\frac{D}{24}}\varphi(\tau,\ffZ) \qquad
{\rm{and}} \qquad
\varphi(\tau,\ffZ+2y)=\nu([0,2y;0])\varphi(\tau,\ffZ)=\varphi(\tau,\ffZ).
$$
The function $\varphi$ is called {\it holomorphic at infinity} if it
has the  Fourier expansion of the following type
\begin{equation}\label{F-exp-phi}
\varphi(\tau,\mathfrak{Z})= \sum_{\substack{n\in \QQ_{\geqslant0},\,
n\equiv \frac{D}{24} \bmod \ZZ,\ l\in \frac 1{2} L^\vee
\\ \vspace{1\jot} 2nt-(l,l)\geqslant 0}}
f(n,l)e^{2i\pi({n}\tau+(l,\mathfrak{Z}))}
\end{equation}
where $L^\vee$ is the dual lattice of the even positive definite
lattice $L$.
This condition is equivalent to the fact that the function $\widetilde{\varphi}$
is holomorphic at the zero-dimensional cusp defined by isotropic vector $f$
in the first copy $U$ in the lattice $L_2=U\oplus U_1\oplus L(-1)$.

Definition \ref{def-JF} suits well for the applications considered in this paper
but we can give another definition which does not depend on the orthogonal
realization of the Jacobi group
$\Gamma_s^J(L)\simeq \SL_2(\ZZ)\rtimes H_s(L)$.
\smallskip

\noindent
{\bf Definition 1.2'.}
{\it A holomorphic function  $\varphi: \HH_1\times (L\otimes \CC)\to \CC$
is called a holomorphic
Jacobi form of weight $k\in \frac{1}{2}\ZZ$ and index $t\in \QQ$
with a character (or a multiplier system) of finite order $\chi: \Gamma^J_s(L)\to \CC^*$
if $\varphi$ satisfies the functional equations \eqref{jacobi-A} and  \eqref{jacobi-H}
and has  a Fourier expansion of type \eqref{F-exp-phi}}.
\smallskip

\noindent
{\bf Remarks.} 1) {\it The classical Jacobi forms of Eichler and Zagier}.
Note that for $n_0=1$, the tube  domain $\cH_{3}(L)$ is isomorphic
to the classical Siegel upper-half space of genus $2$.
If $L\simeq A_1=\latt{2}$ is the lattice $\ZZ$ with quadratic form $2x^2$
and $\chi=\id$ then the definition above is identical
to the definition of Jacobi forms of integral weight $k$ and index $t$ given
in the book \cite{EZ}.

2) The difference between the definitions 1.2 and 1.2' is the
character of the center of the ``orthogonal'' Heisenberg group
$H(L)$. It is more natural to  consider a Jacobi forms $\varphi$
as a modular form with respect to the minimal Jacobi group
$\Gamma_s^J(L)$ and the {\it extended} Jacobi form
$\widetilde{\varphi}(Z)=\varphi(\tau, \ffZ)e^{2\pi i t \omega}$ with
respect to $\Gamma^J(L)$.

\smallskip

We denote the vector space of Jacobi forms  from Definition
\ref{def-JF}'  by ${J}_{k,L;t}(\chi)$ where $\chi=v_\eta^D\times
\nu$ defined by  a character (or a multiplier system) $v_\eta^D$ of
$\SL_2(\ZZ)$ and a binary (or trivial) character $\nu$  of
$H_s(L)$. The space of Jacobi forms from Definition 1.2 is denoted
by $\widetilde{J}_{k,L;t}(v_{\eta}^D\times \widetilde{\nu})$ with evident
modification for the central part of $\widetilde{\nu}$. The character  $\chi$
of $\Gamma^J(L)$ and its restriction
$\widetilde{\chi}=\chi|_{\Gamma_s^J(L)}$ determine each other uniquely
and we
denote both of them by $\chi$, $\widetilde {J}_{k,L;t}(\chi)\simeq
{J}_{k,L;t}(\chi)$ and we will identify these spaces.

A function
$$
\varphi(\tau,
\ffZ)=\sum_{n,l}f(n,l)e^{2i\pi({n}\tau+(l,\mathfrak{Z}))}\in
J_{k,L;t}(\chi)
$$
is called a {\it{Jacobi cusp form}} if $f(n,l)\ne 0$ only if the
hyperbolic norm of its index is positive: $2nt-(l,l)>0$. We denote
this vector space by $J^{cusp}_{k,L;t}(\chi)$. We define the order
of $\varphi$ as follows
\begin{equation}\label{ord-phi}
{\rm Ord} (\varphi)=\min_{f(n,l)\ne 0} (2nt-(l,l)).
\end{equation}
When the character (or the multiplier system) is trivial, we omit it
in the notation of these spaces. If $\chi=v_\eta^D\times \id$, we omit the trivial part.
We see from the definition that $\varphi\equiv 0$ if $t<0$. If $t=0$ then
$J_{k,L;0}(\chi)=M_k(\SL_2(\ZZ), \chi|_{\SL_2})$.
In fact a Jacobi form  corresponds to a vector-valued modular form
of integral or half-integral weight related to the  Weil representation
of the lattice $L(t)$  and  $J_{k,L;t}(\chi)$ is finite dimensional
(see \S 3).
\smallskip

The notation $L(t)$ stands for the lattice $L$ equipped with
bilinear for $t(\cdot ,\cdot )$. We proved above  that if
$J_{k,L;t}(\chi)\ne \{0\}$ then the lattice $L(t)$ is integral.
Any Jacobi form with trivial character  can be  considered as Jacobi
form of  index $1$ (see \cite[Lemma 4.6]{G-K3}). In general we have
the following

\begin{proposition}\label{pr-index}
{\bf 1.} If $L(t)$ is an  even  lattice then
$$
J_{k,L;t}(\chi)=J_{k,L(t);1}(\chi).
$$
If this space is non-trivial then the Heisenberg part  $\nu=\chi|_{H_s(L)}$
of the character is trivial.

{\bf 2.} If $L(t)$ is integral odd (non-even) lattice then
$$
J_{k,L;t}(\chi)=J_{k,L(2t); \frac{1}2}(\chi).
$$
In this case the character $\nu=\chi|_{H_s(L)}$ is of order $2$.

{\bf 3.} If $\varphi(\tau, \ffZ)\in J_{k,L;t}(v_\eta^D\times \nu)$
then $\varphi(\tau, 2\ffZ)\in  J_{k,L;4t}(v_\eta^D\times \id)$.
\end{proposition}
\noindent
{\bf Remark.} This proposition shows that we have to distinguish
in fact only between index $1$ and $\frac{1}2$.
In what follows
we denote by  $J_{k,L}(\chi)$ the space of Jacobi forms of index $1$.

\begin{proof} If $L(t)$ is even then $t(x,x)\in 2\ZZ$ for any $x\in L$.
Therefore $\chi([x,0;0])=e^{i\pi t(x,x)}=1$ and the Heisenberg part
of $\chi$ is trivial. If $L(t)$ is odd then there exists $x\in
L$ such that $t(x,x)$ is odd. Therefore the Heisenberg part of
$\chi$ is non-trivial.

We prove the proposition about the indexes  using a map that we will need in \S 2.
We define an  application for $N\in \QQ_{>0}$
\begin{equation}\label{pi-N}
\pi_{N}: \cH(L(N))\to\cH(L), \qquad
\pi_{N}: \left(\begin{smallmatrix}\omega\\\mathfrak{Z}\\\tau\end{smallmatrix}\right)\mapsto
\left(\begin{smallmatrix}{\omega}/{N}\\\mathfrak{Z}\\\tau\end{smallmatrix}\right).
\end{equation}
This map corresponds to the multiplication
$I_{N}\cdot {\rm{pr}}(Z)$ with $Z\in \cH(L(N))$
and
${\rm{pr}}(Z)\in \cD^\bullet(U\oplus U_1\oplus L(-N))$ where
$I_{N}=\rm{diag}(N^{-1}\Eins_2, \Eins_{n_0}, \Eins_2)$.
We prove the second claim. (The proof of the first one is similar.)

If $\varphi\in {J}_{k,L;t}(\chi)$
then
$$
\widetilde{\varphi}_{1/2}(Z)=\varphi(\tau,\ffZ)e^{\pi i\omega}
=\widetilde{\varphi}\circ \pi_{2t}(Z),\qquad Z\in \cH(L(2t))
$$
is a holomorphic function on $\cH(L(2t))$.
We add index $2t$ to $\{A\}$ and $h$ in order to indicate the  elements of the Jacobi groups
$\Gamma^J(L(2t))$. First we see that
$I_{2t}\{A\}_{2t}I_{2t}^{-1}= \{A\}$.
Therefore
$$
\widetilde{\varphi}_{1/2}|_k \{A\}_{2t}(Z)=\chi(A)\widetilde{\varphi}_{1/2}(Z).
$$
Second we have
$$
I_{2t}[x,y;r]_{2t}I_{2t}^{-1}=[x,y; \frac{r}{2t}]=[x,y;\frac{1}2(x,y)]\cdot
[0,0; \frac{r-t(x,y)}{2t}]\quad
\text{for }\ x,y\in L.
$$
Therefore
$$
\widetilde{\varphi}_{1/2}|_k [x,y;r]_{2t}(Z)=
\widetilde{\varphi}|_k ([x,y;\frac{1}2(x,y)]\cdot[0,0; \frac{r}{2t}-\frac{1}2(x,y)])(\pi_{2t}(Z))=
$$
$$
\chi([x,y;\frac{1}2(x,y)])e^{\pi i (r-t(x,y))}\widetilde{\varphi}_{1/2}(Z)
\qquad\text{where }\
r-t(x,y)\in \ZZ.
$$
It means that $\widetilde{\varphi}_{1/2}$ is an (extended) Jacobi form of
index $\frac{1}2$ with the same $\SL_2(\ZZ)$- and Heisenberg characters
with respect to the even  lattice $L(2t)$.
\end{proof}

In Definitions \ref{def-JF} and \ref{def-JF}' and in the  proof of the last proposition
we used two interpretations of Jacobi forms as a function on $\HH_1\times (L\otimes \CC)$
and on the tube domain $\cH(L)$.
For any $\tau=u+iv\in \HH_1$ and $\ffZ\in L\otimes \CC$
we can find $\omega=u_1+iv_1\in \HH_1$ such that
$2v_1v-(\Im(\ffZ),\Im(\ffZ))>0$ or, equivalently,
${}^t(\omega, {}^t\ffZ, \tau)\in \cH(L)$. In the next lemma
we fix an independent  part of this parameter $\omega$.
\begin{lemma}\label{lem-v-free}
Let $Z
={}^t(\omega, {}^t{\mathfrak{Z}}, \tau)\in\cH(L)$.
Then  the quantity
$$
\widetilde{v}(Z)=v_1-\frac{(\hbox{Im}(\ffZ),\hbox{Im}(\ffZ))}{2v}>0
$$
is invariant with  respect to the action of
the real Jacobi group $\Gamma^{J}(L\otimes \RR)$.
\end{lemma}
\begin{proof}
For any $Z=X+iY\in\cH(L)$ we consider its image
$\cZ=\cX+i\cY=[{\rm pr}(Z)]=[{}^t(-\frac{1}2 (Z,Z)_1, {}^tZ, 1)]$
in the projective homogeneous domain $\cD(L)$ with $L=2U\oplus L(-1)$.
For any $M\in \Orth^+(L)$ we have
$$
(M\cZ, M\overline{\cZ})=(\cZ, \overline{\cZ})=2(\cY,\cY)_1=
2(2v_1\cdot v -(\hbox{Im}(\ffZ), \hbox{Im}(\ffZ)))=4v\cdot \widetilde v(Z).
$$
By the definition of the action of the group $\Orth^+(L)$ on the tube domain
we have
$$
4v\cdot \widetilde v(\cZ)=
2(\cY,\cY)_1=(M\cZ, M\overline{\cZ})=J(M,\cZ)\cdot \overline{J(M, \cZ)}
(M\latt{\cZ}, M\latt{\overline{\cZ}})=
$$
$$
2|J(M,\cZ)|^2 \bigl(\cY(M\latt{\cZ}), \cY(M\latt{\cZ})\bigr)_1=
4|J(M,\cZ)|^2 v(M\latt{\cZ})\cdot \widetilde{v}(M\latt{\cZ})=
$$
$$
4v\cdot \widetilde{v}(M\latt{\cZ}).
$$
\end{proof}

\noindent
{\bf Remark.} The quantity $\widetilde{v}(Z)\in \RR_{>0}$ is a free part of the variable
$Z$
in the extended Jacobi form $\widetilde{\varphi}(Z)$:
$$
(\tau,{}^t\ffZ) \mapsto
{}^t(x_1+i(\widetilde{v}+ \frac{(\hbox{Im}(\ffZ),\hbox{Im}(\ffZ))}{2v}),\, {}^t\ffZ,\ \tau )
\in \cH(L).
$$

The Jacobi forms with respect to a lattice $L$ form a bigraded ring
$J_{*, L;*}$ with respect to weights and indexes.
In the next proposition  we define a direct
(or tensor) product of two Jacobi forms. Its proof follows directly from the definition.
\begin{proposition}\label{pr-prodJ}
Let $\varphi_1\in J_{k_1,L_1;t}(\chi_1\times \nu_1)$
and
$\varphi_2\in J_{k_2,L_2;t}(\chi_2\times \nu_2)$ two Jacobi forms of the same index $t$.
Then
$$
\varphi_1\otimes \varphi_2:=\varphi_1(\tau, \ffZ_1)\cdot\varphi_2(\tau, \ffZ_2)
\in J_{k_1+k_2,L_1\oplus L_2;t}(\chi_1\chi_2\times \nu_1\nu_2)
$$
where $\nu_1\nu_2$ is a character of the group $H(L_1\oplus L_2)$ defined by
$$
(\nu_1\nu_2)([x,y;r])=\nu_1([x_1,y_1; \tfrac{1}{2}(x_1,y_1)])
\nu_2([x_2,y_2;\tfrac{1}{2}(x_2,y_2)])e^{i\pi t((x_1,y_1)+(x_2,y_2)+2r)}
$$
for $[x,y;r]=[x_1\oplus x_2,y_1\oplus y_2;r]\in H(L_1\oplus L_2)$.
The tensor product of two Jacobi forms is a cusp form
if at least one of them is a cusp form.
\end{proposition}

It is known (see \cite{G1}) that the space $J_{k,L;t}(\chi)$ is trivial
if $k< \frac{n_0}2$ where $\rank L=n_0$.
 The minimal possible weight $k=\frac{n_0}2$ is called
{\it singular weight}.
For any Jacobi form $\varphi$ of singular weight the hyperbolic norm
of the index of a non-zero Fourier coefficient $f(n,l)$
(see \eqref{F-exp-phi})
is equal to zero: $2nt-(l,l)=0$,
i.e. ${\rm Ord} (\varphi)=0$.
\smallskip

\noindent
{\bf Example 1.6.} {\it The Jacobi theta-series}.
The Jacobi theta-series of characteristic $(\frac{1}2, \frac{1}2)$
is  defined by
\begin{equation}\label{theta}
\vartheta(\tau,z)=\sum_{n\in\ZZ}\left(\frac{-4}{n}\right)q^{\frac{n^2}{8}}
r^{\frac{n}{2}}
=-q^{1/8}r^{-1/2}\prod_{n\geqslant 1}\,(1-q^{n-1} r)(1-q^n r^{-1})(1-q^n)
\end{equation}
where $q=e^{2\pi i \tau}$, $\tau\in \HH_1$ and $r=e^{2\pi i z}$, $z\in \CC$.
This is the simplest example of Jacobi form of half-integral index.
The theta-series $\vartheta$  satisfies two functional equations
$$
\vartheta(\tau, z+x \tau+y)
=(-1)^{x +y}e^{-\pi i (x^2\tau+2x z)}\vartheta(\tau, z),
\quad (x, y) \in \ZZ^2
$$
and
$$
\vartheta(A\latt{\tau}, z)=
v_\eta^3(A)(c\tau+d)^{\frac 1{2}}e^{\pi i \frac{cz^2}{c\tau + d}}\vartheta(\tau, z),
\quad
A=\left(\smallmatrix a&b\\ c&d \endsmallmatrix\right)\in \SL_2(\ZZ)
$$
where $v_\eta$ is the multiplier system of the Dedekind $\eta$-function.
Using our notations, $A_1 \simeq \latt{2}$, we have
$$
\vartheta(\tau,z)\in  J_{\frac{1}{2},A_1;\frac{1}{2}}(v_{\eta}^{3}\times v_{H})
$$
where, for short, $v_{H(A_1)}=v_H$ is defined by:
$$
v_H([x,y;r])=(-1)^{x+y+xy+r},\qquad [x,y;r]\in H(\ZZ),\quad
x,y,r\in \ZZ.
$$
The Jacobi theta-series $\vartheta$  is the Jacobi form of singular weight
$\frac{1}2$  with a non-trivial character of the Heisenberg group.
This Jacobi form was not mentioned in \cite{EZ} but it plays an important
role in the construction of the basic Jacobi forms
and reflective modular forms (see \cite{GN3}, \cite{G-EG}, \cite{G10}).
We remind that
$$
\div \vartheta(\tau, z) = \{h\latt{z=0}\,|\, h\in H(\ZZ)\}=
\{z=x\tau+y\,|\, x ,y \in \ZZ\}.
$$
The Jacobi theta-series $\vartheta$ having the triple product formula
\eqref{theta} will be the first main function in our construction of
Jacobi forms for orthogonal lattices.
\smallskip

\noindent
{\bf Example 1.7.} {\it The Jacobi forms of singular weight for $mA_1$}.
Using the Jacobi theta-series we can  construct Jacobi forms of singular weight for $mA_1=\latt{2}\oplus \dots \oplus \latt{2}$.
The tensor product of $m$ Jacobi theta-series is a Jacobi form of singular weight
and index $\frac{1}2$ for $mA_1$:
\begin{equation}\label{theta-Am}
\vartheta_{mA_1}(\tau, z_1,\dots,z_m)=
\prod_{1\leqslant j \leqslant m}\vartheta(\tau,z_j)
\in J_{\frac{m}{2},mA_1;\frac{1}{2}}(v_{\eta}^{3m}\times v_H^{\otimes m})
\end{equation}
where
$$
v_H^{\otimes m}([x,y;r])=v_{H(mA_1)}([x,y;r])=(-1)^{r+\sum_{j=1}^{m}x_j+y_j+x_jy_j}
$$
for any $x_j$, $y_j$, $r$ in $\ZZ$. For even $m$ we can construct
Jacobi forms of singular weight and index $1$ because
\begin{equation}\label{theta-D2}
\vartheta_{2A_1}^{(1)}(\tau, z_1,z_2)=\vartheta(\tau,z_1+z_2)\cdot
\vartheta(\tau,z_1-z_2)
\in J_{1,2A_1}(v_{\eta}^{6}).
\end{equation}
Taking different orthogonal decompositions of the lattices $8A_1$
we obtain $105$ Jacobi forms of weight $4$  and index $1$
with trivial character.
\smallskip

\noindent
{\bf Example 1.8.} {\it The Jacobi forms of singular weight for $D_m$}.
We recall the definition of the  even quadratic lattice $D_m$.
(In this paper we denote by $A_m$, $D_m$, $E_m$ the lattices
generated by the corresponding root systems).
We use the standard Euclidian  basis $\latt {e_i}_{i=1}^m$ ($(e_i,e_j)=\delta_{ij}$)
in $\ \ZZ^m$.
Then
$$
D_m=\{(x_1,\dots,x_m)\in \ZZ^m\,|\, x_1+\dots+x_m\in 2\ZZ\}\quad
(m\geqslant1)
$$
is the maximal even sublattice in $\ZZ^m$.
The theta-product \eqref{theta-Am} is a Jacobi form of index $1$ for
$D_m$ with trivial Heisenberg character because the quadratic form
$x_1^2+\dots+x_m^2$ is even on $D_m$
\begin{equation}\label{theta-Dm}
\vartheta_{D_m}(\tau,\mathfrak{Z}_m)=
\vartheta(\tau,z_1)\cdot \ldots \cdot \vartheta(\tau,z_m)
\in J_{\frac{m}{2},D_m}(v_{\eta}^{3m}).
\end{equation}
We note that $D_2\cong A_2$ and
$\vartheta_{D_2}=\vartheta_{2A_1}^{(1)}$.
For the lattice $D_4$ we can give two more examples:
\begin{multline}\label{theta-D4a}
\vartheta_{D_4}^{(2)}(\tau,\mathfrak{Z}_4)=
\vartheta(\tau,\frac{-z_1+z_2+z_3+z_4}2)
\vartheta(\tau,\frac{z_1-z_2+z_3+z_4}2)\\
\vartheta(\tau,\frac{z_1+z_2-z_3+z_4}2)
\vartheta(\tau,\frac{z_1+z_2+z_3-z_4}2)
\in J_{2,D_4}(v_{\eta}^{12})
\end{multline}
and
\begin{multline}\label{theta-D4b}
\vartheta_{D_4}^{(3)}(\tau,\mathfrak{Z}_4)=
\vartheta(\tau,\frac{z_1+z_2+z_3+z_4}{2})
\vartheta(\tau,\frac{z_1+z_2-z_3-z_4}{2})\\
\vartheta(\tau,\frac{z_1-z_2-z_3+z_4}{2})
\vartheta(\tau,\frac{z_1-z_2+z_3-z_4}{2})
\in J_{2,D_4}(v_{\eta}^{12}).
\end{multline}
Analyzing the divisors of the Jacobi forms we obtain
the relation
$$
\vartheta_{D_4}(\tau,\mathfrak{Z}_4)=
\vartheta_{D_4}^{(2)}(\tau,\mathfrak{Z}_4)+
\vartheta_{D_4}^{(3)}(\tau,\mathfrak{Z}_4).
$$

{\bf Jacobi forms and the  Weil representation.}
The Jacobi forms  can be considered
as vector valued $SL_2(\ZZ)$-modular forms (see \cite{B3},
\cite{KP}, \cite{G-K3}, \cite{Sch}, \cite{Sk}) related to the Weil
representation.
To compare the examples considered above with vector-valued
modular forms we recall the definitions from \cite{G-K3}
for Jacobi form of index one.

Let $L$ an even positive definite lattice of rank $n_0$ and
$$
\varphi(\tau,\mathfrak{Z})=
\sum_
{\substack{n\in \ZZ,\ l\in L^\vee
\\ \vspace{0.5\jot} 2n-(l,l)\geqslant 0}}
f(n,l)e^{2i\pi({n}\tau+(l,\mathfrak{Z}))}
\in J_{k,L}
$$
a Jacobi form of weight $k$ and index one.
By  $q=q(L)$ we denote  the level of the lattice $L$,
i.e. the smallest integer such that  $L^\vee(q)$ is an even lattice.
Then we have the following  representation
(see \cite[Lemma 2.3]{G-K3} with $m=1$)
$$
\varphi(\tau,\mathfrak{Z})=
\sum_{\mu\in D(L)}
\phi_\mu(\tau)\theta^L_{\mu}(\tau,\mathfrak{Z})
$$
where
$$
\phi_\mu(\tau)=
\sum_
{\substack{ r\geqslant0\\ \vspace{0.5\jot}
\frac{2r}q \equiv -(\mu,\mu)\,{\rm mod}\, 2\ZZ}}
f_h(r) \exp{(2\pi i\, \frac{r}{q}\tau)},
\qquad
f_\mu(r)=f(\tfrac{2r+(\mu,\mu)}{2q},\mu)
$$
and
$$
\theta^L_{\mu}(\tau,\mathfrak{Z})=\sum_{l\in \mu+L}
e^{i\pi m((l,l)\tau+2(l,\mathfrak{Z}))}
$$
is  the theta-series  with characteristic $\mu\in D(L)=L^\vee/L$
where $L^\vee$ is the dual lattice.
For any matrix
$M=(\smallmatrix a&b\\c&d\endsmallmatrix)\in SL_2(\ZZ)$
the theta-vector
$
\Theta_L(\tau,\mathfrak{Z})
=\bigl(\theta^L_{\mu}(\tau,\mathfrak{Z})\bigr)_{\mu\in D(L)}
$
has the following transformation property
$$
\Theta_L
(\frac{a\tau+b}{c\tau+d},\frac{\mathfrak{Z}}{c\tau+d})
=(c\tau+d)^{\frac {n_0}2}
U(M)
\exp(\frac{\pi i c(\mathfrak{Z}, \mathfrak{Z})}{c\tau+d})
\,\Theta_L(\tau,\mathfrak{Z})
$$
where
$U(M)$ is a unitary   matrix.  In particular,
for
$T=\left(\smallmatrix 1&1\\0&1\endsmallmatrix\right)$ and
$S=\left(\smallmatrix 0&-1\\1&\ \,0\endsmallmatrix\right)$ we have
$$
U(T)={\rm diag\,}(e^{i\pi (\mu,\mu)})_{\mu\in D(L)},
\quad
U(S)=(-i)^{\frac{n_0}{2}}(\sqrt{|D(L)|})^{-\frac{1}2}
\bigl(
e^{-2i\pi(\mu,\nu)}\bigr)_{\mu,\nu\in D(L)}.
$$
Therefore  $\Phi(\tau)=(\phi_\mu(\tau))_{\mu\in D(L)}$
is a holomorphic vector-valued modular form of weight
$k-\frac{n_0}2$ for the conjugated representation $\overline{U}(M)$
of $SL_2(\ZZ)$. In particular, the weight of any holomorphic
Jacobi  form  is greater or equal to
$\frac{n_0}2$ (see \cite{G1}).
We note also that the  theta-series
$\theta^L_{\mu}$  are linear independent
and $\Theta_L$
is invariant with respect to the action
of  the stable orthogonal group $\widetilde{\Orth}(L)$
(see Theorem 2.2 below)
and, in particular, with respect to the Weyl group
of the lattice $L$, $W_{2}(L)$,  which is a subgroup
of $\widetilde{\Orth}(L)$ generated by $2$-reflections
in the lattice $L$.

We get the simplest example for an even unimodular
lattice $N$ of rank $n_0$
 (see \cite{G1} and \cite[Lemma 4.1]{G-K3})
 \begin{equation}\label{J-unimodular}
 \Theta_N(\tau, \mathfrak{Z})=
 \sum_{l\in N}
 e^{\pi i (l,l)\tau+ 2\pi i (l, \mathfrak{Z})}
 \in J_{\frac{n_0}2, N}.
 \end{equation}
Moreover we have that two linear spaces of Jacobi forms are
isomorphic
$$
	{J}_{k_1, L_1}\cong{J}_{k_1+\frac{n_2-n_1}2, L_2}
$$
if $L_1$ (rank\,$L_1=n_1$)  and  $L_2$ (rank\,$L_2=n_2$)
are two lattices with isomorphic discriminant forms
(see \cite[Lemma 2.4]{G-K3}).
\smallskip

\noindent
{\bf Example 1.9.}  {\it The Weil representation for $D_m$}.
We recall that  $|D_{m}^\vee/D_{m}|=4$ and
$$
D_m^\vee/D_m=\{\mu_i,\ i\bmod 4\}=
\{0,\frac{1}2(e_1+\dots+e_m),\,  e_1,\,  \frac{1}2(e_1+\dots
+e_{m-1}-e_m)\ \bmod D_m\}
$$
is the cyclic group of order $4$ generated by
$\mu_1=\frac{1}2(e_1+\dots+e_m)\bmod D_m$, if $m$ is odd,
and the product of two groups of order $2$, if $m$ is even.
We have the following  matrix
of inner products in the discriminant group of $D_m$
of the non-trivial classes modulo $D_m$
$$
\bigl((\mu_i,\mu_j)\bigr)_{i,j\ne 0}=
\left(\begin{matrix}
\frac{m}4&\frac{1}2&\frac{m-2}4\\
\frac{1}2&1&\frac{1}2\\
\frac{m-2}4&\frac{1}2&\frac{m}4
\end{matrix}\right)\qquad  (\mu_i\in D_m^\vee/D_m)
$$
where the diagonal elements are taken modulo $2\ZZ$
and the non-diagonal elements are taken modulo $\ZZ$.
We note that the discriminant group of $D_m$ depends only on $m\bmod 8$.
This gives the formula for $U(T)$ and $U(S)$.

1) For $m\equiv 4\bmod 8$, we have
$$
U(T)={\rm diag\,}(1,-1,-1,-1),\quad
U(S)=-\frac{1}{2}\left(\begin{smallmatrix}
1&\ \,1&\ \,1&\ \,1\\
1&\ \,1&-1 &-1\\
1&-1  &\ \,1&-1\\
1&-1  &-1&\ \,1
\end{smallmatrix}\right).
$$
We put $\theta^{D_m}_{i}(\tau,\mathfrak{Z_m})
:=\theta^{D_m}_{\mu_i}(\tau,\mathfrak{Z_m})$ for  $i\bmod 4$.
The matrices  $U(T)$ and $U(S)$ have the  three common eigenvectors:
$$
\theta^{D_m}_{1}-\theta^{D_m}_{3},
\quad \theta^{D_m}_{1}-\theta^{D_m}_{2},
\quad \theta^{D_m}_{2}-\theta^{D_m}_{3}
\qquad (m\equiv 4\bmod 8).
$$
If $m=4$ we get the  Jacobi forms
$\vartheta_{D_4}$, $\vartheta_{D_4}^{(1)}$
and $\vartheta_{D_4}^{(2)}$ obtained above as theta-products.

2) For $m\equiv 0\bmod 8$, we have
$$
U(T)={\rm{diag}}(1,1,-1,1),\quad
U(S)=
\frac{1}{2}\left(\begin{smallmatrix}
1 & \ \,1 &\ \, 1 & \ \,1 \\
1 & \ \,1 & -1 & -1 \\
1 & -1 & \ \,1 & -1\\
1 & -1 & -1 & \ \,1
\end{smallmatrix}\right).
$$
These lattices have again two linear independent  common eigenvectors.
The first one is the theta-product
$\vartheta_{D_m}=\theta^{D_m}_{1}-\theta^{D_m}_{3}$.
The second eigenvector
$\vartheta_{D_m^+}=\theta^{D_m}_{0}+\theta^{D_m}_{1}$
is equal to the Jacobi theta-series of  the unimodular lattice
$D_m^+=\latt{D_m,\, \mu_1}$. In particular, $D_8^+=E_8$
and $\vartheta_{D_m^+}=\Theta_{E_8}$ (see \eqref{J-unimodular}).
To understand better the role of the Jacobi theta-series
$\vartheta$ we consider one more case.

3) For  $m\equiv 1 \bmod 8$, we have
$$
U(T)={\rm{diag}}(1,e^{i\frac{\pi}{4}},-1,e^{i\frac{\pi}{4}}), \quad
U(S)=\frac{1}{2}e^{-i\frac{\pi}{4}}
\left(\begin{smallmatrix}
1 & \ \,1 & \ \,1 & \ \,1 \\
1 & -i & -1 & \ \,i\\
1 & -1 & \ \,1 & -1\\
1 & \ \,i  & -1 & -i
\end{smallmatrix}\right)
$$
and  $\vartheta_{D_m}=\theta^{D_m}_{1}-\theta^{D_m}_{3}$
is the only Jacobi form of singular weight.
Moreover, for $m=1$ we get
$$
\vartheta_{D_1}(\tau,z)
\in J_{\frac{1}{2},\latt{4};1}(v_{\eta}^3)=J_{\frac{1}{2},A_1;2}(v_{\eta}^3)
=J_{\frac{1}{2},2}(v_{\eta}^3).
$$
The last space is the space of classical Jacobi forms of weight $\frac{1}{2}$, index $2$  with the multiplier system $v_{\eta}^3$.
It is easy to check that
$\vartheta_{D_1}(\tau,0)=\vartheta_{D_1}(\tau,\frac{1}{2})=0$.
Therefore
$$
\vartheta_{D_1}(\tau,z)=\vartheta(\tau,2z).
$$

4) Analyzing $U(T)$ and $U(S)$ for all other $m$ modulo $8$
we get only one  common eigenvector corresponding to the theta-product
$\vartheta_{D_m}=\theta^{D_m}_{1}-\theta^{D_m}_{3}$.
Therefore Example 1.8 contains all possible Jacobi forms
of singular  weight (and index one) for $D_m$.
\smallskip

\noindent
{\bf Example 1.10.}  {\it The lattice  $E_6$}.
Let $E^\vee_6$ be the dual lattice of $E_6$ and  $D(E_6)$
its discriminant group. We have
$$
D(E_6)=E^\vee_6/E_6\simeq \ZZ/3\ZZ
\quad{\rm and }\quad
q_{D(E_6)}=-q_{D(A_2)}.
$$
The  discriminant group has the following
system of representatives (see \cite{Bou}, Planche V):
$D(E_6)=\left\{0,\mu, 2\mu\right\}$
where
$\mu^2\equiv \frac{4}3\mod 2\ZZ$.
We have
\[
U(T)={\rm{diag}}\left(1,\rho^2,\rho^2\right),\quad
U(S)=\frac{i}{\sqrt{3}}\left(\smallmatrix 1 &1 &1 \\ 1 & \rho^2 & \rho \\ 1 & \rho & \rho^2 \endsmallmatrix\right)
\]
with $\rho=e^{\frac{2i\pi}{3}}$.
We get
\[
\theta_{E_6}(\tau,\mathfrak{Z}_6)
=(\theta_{1}-\theta_{2})(\tau,\mathfrak{Z}_6)\in J_{3,E_6}(v_{\eta}^{16}).
\]
This Jacobi form is invariant with respect to the Weyl group
$W(E_6)$.

The simple construction of Jacobi forms using products of  Jacobi
theta-series   has a lot of advantages.
First, we get Jacobi forms of singular weight with
a very simple divisor.
Second, we can easily determine the maximal group
of symmetries with respect to the abelian variable.
This fact is important in the next section
in which we construct
modular forms of singular weight with respect to
orthogonal groups.
\smallskip

\noindent
{\bf Example 1.11}
{\it The Jacobi theta-series  $\vartheta_{3/2}$.}
We can get more examples using the second theta-series
of weight $1/2$ and index $3/2$ with respect
to the full modular group $SL_2(\ZZ)$. This function is  related
to twisted affine Lie algebras and is important in the construction
of basic reflective Siegel modular forms (see \cite{GN3})
\begin{equation}\label{theta3/2}
\vartheta_{3/2}(\tau,z)=\frac{\eta(\tau)\vartheta(\tau,2z)}{\vartheta(\tau,z)}
\in J_{\frac{1}2, A_1;\frac{3}2}(v_\eta\times v_H)=
J_{\frac{1}2, \latt{6};\frac{1}2}(v_\eta\times v_H)
\end{equation}
which is  given by the quintiple product formula
\begin{align*}\
\vartheta_{3/2}(\tau,z)&=\sum_{n\in\ZZ}\left(\frac{12}{n}\right)q^{\frac{n^2}{24}}
r^{\frac{n}{2}}=\\
{}&q^{\frac{1}{24}}r^{-\frac{1}{2}}
\prod_{n\geqslant 1}\,(1+q^{n-1} r)(1+q^{n} r^{-1})
(1-q^{2n-1} r^{2})(1-q^{2n-1} r^{-2})(1-q^n).
\end{align*}
We have
\begin{equation}\label{theta-mA3}
\vartheta_{mA_1(3)}(\tau, z_1,\dots,z_m)
=\prod_{1\leqslant j \leqslant m}\vartheta_{3/2}(\tau,z_j)
\in J_{\frac{m}{2},m\latt{6};\frac{1}{2}}(v_{\eta}^{m}\times v_H^{\otimes m})
\end{equation}
(we recall that $m\latt{6}=mA_1(3)$
denotes the orthogonal sum of $m$ copies
of the lattice $\latt{6}$ of rank one).
The same theta-product can be considered as a Jacobi form of index
$1$ for the lattice $D_m(3)$
and
\begin{equation}\label{theta-Dm3}
\vartheta_{D_m(3)}(\tau,\mathfrak{Z}_m)=
\vartheta_{3/2}(\tau,z_1)\cdot \ldots \cdot \vartheta_{3/2}(\tau,z_m)
\in J_{\frac{m}{2},D_m(3)}(v_{\eta}^{m})
\end{equation}
where $D_m(3)$ is the lattice $D_m$ renormalized by $3$.
In this simple way, we construct  examples of Jacobi forms
of singular weight with trivial character for even $n_0\geqslant8$:
$D_8$, $D_7\oplus D_3(3)$, $D_6\oplus D_6(3)$,
$D_5\oplus D_9(3)$ and so on (see Proposition \ref{Form-sing}).

\section{The lifting of Jacobi forms of half-integral index}

The lifting of the  Jacobi form $\vartheta_{D_8}$ (see \eqref{theta-Dm})
is a reflective modular form with respect to the orthogonal  group
$\Orth^+(2U\oplus D_8(-1))$ (see \cite{G10}) which is equal to the Borcherds--Enriques automorphic discriminant $\Phi_4$
of the moduli space of the Enriques surfaces introduced in
\cite{B2}. The lifting of the  Jacobi form
$$
\eta^9(\tau)\vartheta_{D_5}(\tau,\ffZ_5)\in J_{7,D_5}
$$
determined the unique canonical differential form on the modular variety
of the orthogonal group $\widetilde{\SO}^+(2U\oplus D_5(-1))$
having  Kodaira dimension $0$. In \cite{G10} there were found three
such modular varieties of dimensions $4$, $6$ and $7$.
The cusp form of  the modular variety of dimension $4$ is defined
by a Jacobi form of half-integral index with a character of order
$2$ (see Example 2.4 below).
In this  section we give a variant of the lifting
of  Jacobi forms of half-integral index with a character.
This theorem is a generalization of Theorem 3.1 in \cite{G-K3}
(the case of Jacobi forms of orthogonal type with trivial character) and
Theorem 1.12 in \cite{GN3} (the case of Siegel modular forms with respect
to a paramodular group of genus $2$).
All these constructions are particular cases of Borcherds additive lifting (see \cite[\S 14]{B3}) of vector valued modular forms.
Nevertheless  our  approach related to  Jacobi forms gives
in a natural way many new important examples
of reflective modular forms  for orthogonal groups.
Theorem \ref{thm-lift} is a necessary tool for this purpose.

We can define a Hecke operator which multiplies the index of  Jacobi
forms. This operator is similar to the operator $V_m$ of \cite{EZ}
or to the `minus'-Hecke operator introduced in \cite{G-K3}--\cite{G-Ab}
in the case of Siegel modular forms of arbitrary genus
or for the modular forms for orthogonal groups. We apply such operators
to elements of $J_{k,L;t}(v_{\eta}^{D}\times \nu)$ where $\nu$ is
a binary character of the minimal integral Heisenberg group
$H_s(L)$

\begin{proposition}\label{pr-hecke}
Let $\varphi\in J_{k,L;t}(v_\eta^D\times \nu)$ not identically zero.
We assume that   $k$ is integral,
$t$  is  rational and  $D$ is an even divisor of $24$.
If $Q=\frac{24}{D}$ is odd we assume that
the character of the minimal integral Heisenberg group
$\nu: H_s(L)\to \{\pm 1\}$ is  trivial.
Then for any natural $m$ coprime to $Q$ the function
$$
\varphi\vert_{k,t}T_-^{(Q)}(m)(\tau, \ffZ)=
\sum_{\substack{ad=m, \ a >0 \\ \vspace{0.5\jot} b \bmod d}}
a^{k}v_\eta^D(\sigma_a)
\varphi(\frac{a\tau+bQ}{d}, \,a\ffZ),
$$
where $\sigma_a\in\SL_2(\ZZ)$ such that
$\sigma_a\equiv\left(\begin{smallmatrix}a^{-1}& 0\\0& a\end{smallmatrix}\right)\bmod Q$,
belongs to $J_{k,L;mt}(v_{\eta,m}^D\times \nu)$.
The new  $\SL_2(\ZZ)$-character is defined as follows:
$$
v_{\eta,m}^D(A)=v_\eta^D(A_m)\qquad {\rm for \ all}\ A\in\SL_2(\ZZ)
$$
with $A_m\left(\begin{smallmatrix}1 & 0\\ 0 & m\end{smallmatrix}\right)\equiv
\left(\begin{smallmatrix}1 &0\\0&m\end{smallmatrix}\right)A\bmod Q$.
The character $v_{\eta,m}^D$ depends only on $m\bmod Q$.
\end{proposition}

\begin{proof}
It is known that $\Ker v_\eta^D$
contains the principle congruence subgroup $\Gamma(Q)<\SL_2(\ZZ)$
(see \cite[Lemma 1.2]{GN3}).
We consider the following subgroup
$
\Gamma^J(Q)\simeq\Gamma(Q)\rtimes \Ker(\nu)
$
of the Jacobi group. We identify it with
the  corresponding parabolic subgroup in the orthogonal group
$\Orth^{+}(2U\oplus L(-1))$.
For $(m,Q)=1$, let
$$
T^{(Q)}(m)=\Gamma(Q)\sum_{\substack{ad=m, \ a >0 \\b \bmod d}}
\sigma_a \begin{pmatrix}a & bQ\\0& d\end{pmatrix}
$$
be the usual Hecke operator for $\Gamma(Q)$.
We associate to the element $T^{(Q)}(m)$ the element $T_-^{(Q)}(m)$
of the Hecke ring of the parabolic subgroup (see \cite{G-K3} and \cite{GN3})
$$
T_-^{(Q)}(m)=\Gamma^J(Q)\sum_{\substack{ad=m, \ a >0 \\b \bmod d}}  \{\sigma_a\}M_{a,b,d}
$$
where
$M_{a,b,d}=\diag\bigl(
\begin{pmatrix}  a&-bQ\\0&d
\end{pmatrix},\  \Eins_{n_0},\
m^{-1}\begin{pmatrix}  a& bQ\\0&d
\end{pmatrix}
\bigr)$.
This is a sum of some double cosets with respect to $\Gamma^J(Q)$.
We consider the extended Jacobi form
$\widetilde{\varphi}(Z)=\varphi(\tau,\mathfrak{Z})e^{2i\pi t\omega}$
with $Z={}^t(\omega, {}^t\ffZ,\tau) \in\cH(L)$
which is modular with respect to the parabolic subgroup.
Then we have
$$
\widetilde{\psi}(Z)=(\widetilde{\varphi}\vert_k T_-^{(Q)}(m))(Z)=
\sum_{\substack{ad=m, \ a>0\\b \bmod d}}(\widetilde{\varphi}
\vert_k \{\sigma_a\}M_{a,b,d})(Z).
$$
By definition, we have
$$
(\widetilde{\varphi}\vert_k \{\sigma_a\} M_{a,b,d})(Z)
=a^k v_{\eta}^D(\sigma_a)\varphi(\frac{a\tau+bQ}{d},a\mathfrak{Z})e^{2i\pi mt\omega}.
$$
Therefore the Hecke operator of the proposition corresponds to
the Hecke operator $T_-^{(Q)}(m)$ of the parabolic subgroup $\Gamma^J(Q)$
acting on the modular forms with respect to the parabolic subgroup $\Gamma^J(Q)$.
We remark that the new index of the extended function on $\cH(L)$
is equal to $mt$.
The case of  modular transformations is similar to the theory of usual  Hecke
operators (see \cite{Sh}).
If $A\in\SL_2(\ZZ)$, then  somewhat lenghty but easy calculations give us
$$
\widetilde{\psi}\vert_k \{A\}
=\sum_{\substack{a'd'=m,\ a'>0\\b' \bmod d'}}
(\widetilde{\varphi}\vert_k\{A_m\})\vert_k \{\sigma_{a'}\}M_{a',b',d'}
=v_{\eta}^D(A_m)\widetilde{\psi}.
$$
This is due to the fact that the group
$\Gamma(Q)$ is normal in $\SL_2(\ZZ)$ and then
$$
\Gamma^J(Q)\{A_m\}^{-1}\{\sigma_a\}M_{a,b,d}\{A\}
\neq \Gamma^J(Q)\{A_m\}^{-1}\{\sigma_{a'}\}M_{a',b',d'}\{A\}
$$
for distinct $a$ and $a'$ prime to $Q$.
Secondly we consider the abelian transformations.
Let $h=[x,y;r]\in H_s(L)$. Then we have
$$
\widetilde{\psi}\vert_k h
=\sum_{\substack{ad=m, \ a>0\\b \bmod d}}\nu (h'_{a,b,d})\widetilde{\varphi}
\vert_k\{\sigma_a\}M_{a,b,d}
$$
where
$h'_{a,b,d}=\left\{\sigma_a\right\}(M_{a,b,d}\cdot h)\left\{\sigma_a^{-1}\right\}=[x',y';r']$ and
$$
[x',y';r']
=[(\delta d+\gamma bQ)x-a\gamma y, -(\beta d+\alpha bQ) x+\alpha ay;mr]\in H(L)
$$
with $\sigma_a=\left(\begin{smallmatrix}\alpha& \beta\\ \gamma &
\delta\end{smallmatrix}\right)
\equiv\left(\begin{smallmatrix}a^{-1} & 0\\0& a\end{smallmatrix}\right)\bmod Q$.
We note that
$$
(x',y')\equiv \bigl(m(\alpha\delta+\beta\gamma)+2\alpha \gamma ab Q\bigr)(x,y)
\equiv m(x,y) \bmod 2s(L).
$$
Therefore if $\nu=\id$ then
$$
\nu([x',y';r'])=e^{2\pi i t(mr-\frac{1}2 m(x,y))}
$$
because $t\cdot s(L)\in \ZZ$. This proves the formula for odd $Q$.
If $Q$ is even we have
$[x',y';r']=[mx+Q\widetilde{x}, y+Q\widetilde{y};mr]$.
Then
$$
[-Q\widetilde{x},-Q\widetilde{y}; -\frac{Q^2}2(\widetilde{x},\widetilde{y})]\cdot [x',y';r']
=[mx, y;mr+\frac{Q}2\bigl(
-(\widetilde{x},y')+m(\widetilde{y},x')-Q(\widetilde{x},\widetilde{y})\bigr)].
$$
As $Q$ is even we have
$\nu([-Q\widetilde{x},-Q\widetilde{y}; -\frac{Q^2}2(\widetilde{x},\widetilde{y})])=1$.
But in this case $m=2m_0+1$ is odd so
$$
\nu([x',y';r'])=\nu([mx, y;mr])=\nu([x, y; mr-m_0(x,y)])=\nu([x, y;r]).
$$
We calculate the Fourier expansion
of $\varphi\vert_{k,t}T_-^{(Q)}(m)$ in the proof of Theorem \ref{thm-lift}
(see below). It shows that it is a holomorphic Jacobi form.
\end{proof}

Let $L$ be an even lattice.
{\it The stable orthogonal group} $\widetilde{\Orth}(L)$
is the subgroup of $\Orth^+(L)$ whose elements induce the identity
on the discriminant group $D(L)=L^\vee/L$
$$
\widetilde{\Orth}(L)=\{g\in \Orth(L) {\rm{\ such \ that \ }}
\forall\  l\in L^\vee\,:\, \ g(l)-l \in L\}.
$$
\begin{theorem}\label{thm-lift}
Let $\varphi\in J_{k,L;t}(v_\eta^D\times \nu)$, $k$ be  integral,
$t$  be  rational, $D$ be an even divisor of $24$.
If the conductor $Q=\frac{24}D$ is odd we assume that $\nu$ is trivial.
Fix $\mu\in(\ZZ/Q\ZZ)^*$. Then the function
$$
\Lift_{\mu}(\varphi)(Z)=f(0,0)E_k(\tau)+
\sum_{\substack{m\equiv \mu \bmod Q\\ m\geqslant 1}}
m^{-1}(\widetilde{\varphi}\vert_k T_-^{(Q)}(m))\circ \pi_{Qt}(Z),
$$
is a modular form of weight $k$ with respect to the stable orthogonal group
$\widetilde{\Orth}^+(2U\oplus L(Qt))$ of the even lattice $L(Qt)$
with a character of order $Q$ induced by
$v_{\eta,\mu}^D$, the binary Heisenberg character $\nu$
of $H_s(L(Qt))$ and the character $e^{2i\pi\frac{\mu}{Q}}$
of the center of $H(L(Qt))$.
In the formula above $f(0,0)$ is the zeroth Fourier coefficient of $\varphi$,
$E_k$ is the Eisenstein  series of weight $k$ with respect to $\SL_2(\ZZ)$
and the map $\pi_{Qt}$ was defined in \eqref{pi-N}.
\end{theorem}
\begin{proof}
{\bf The Eisenstein series} $E_k$.
First we note that  $f(0,0)$ could be non-zero only for the trivial character
$v_\eta^D=\id$. In this case
$\varphi(\tau, 0)=f(0,0)+ \dots$ is a non-zero  modular form
of weight $k$ with respect to $\SL_2(\ZZ)$.
Therefore $k\geqslant 4$ and $E_k$ is well defined.
We note that $E_k$ is a Jacobi form of index $0$.
\smallskip

{\bf The lattice $L(Qt)$.} This lattice  is even   for all $Q$.
The lattice $L(t)$ is integral for a non zero Jacobi form $\varphi$.
If $Q$ is odd then $L(t)$ is even because the character $\nu$
is trivial in this case (see Proposition \ref{pr-index}).
Therefore for all $Q$ the lattice $L(Qt)$ is even.
\smallskip

{\bf The character of} $\Gamma^J(L(Qt))$.
According to Proposition \ref{pr-hecke}
$$
\varphi_m(\tau,\mathfrak{Z})=(\varphi\vert_{k,t} T_-^{(Q)}(m))(\tau,\mathfrak{Z})
\in J_{k,L;mt}(v_{\eta,\mu}^D \times \nu).
$$
We can defined an extended Jacobi form using the map $\pi_{Qt}$ (see \eqref{pi-N}).
According to Proposition \ref{pr-index}
$$
\varphi_m(\tau,\mathfrak{Z})e^{2i\pi\frac{m}{Q}\omega}\in
\widetilde{J}_{k, L(Qt); \frac m{Q}}(v_{\eta,\mu}^D \times \nu)
$$
is a modular form of weight $k$ with respect to the parabolic
subgroup $\Gamma^J(L(Qt))$ of the orthogonal group
$\Orth^+(2U\oplus L(-Qt))$. We note that the character
$\nu$ of the minimal integral Heisenberg group $H_s(L(Qt))$ is extended
to the center of $H(L(Qt))$ by the formula
$$
\nu([0,0;r])=e^{2\pi i \frac{m}{Q} r}= e^{2\pi i \frac{\mu}{Q} r}.
$$
If $f(0,0)\ne 0$ then $v_\eta^D=\id$, i.e. $D=24$,  $Q=1$ and  $\nu=\id$.
Therefore all terms in the sum defining the lifting $\Lift_{\mu}(\varphi)$
have the same character with respect to
$\Gamma^J(L(Qt))<\Orth^+(2U\oplus L(-Qt))$.
\smallskip

{\bf Convergence.}
Let $Z={}^t(\omega, {}^t\ffZ,\tau)\in \cH(L(Qt))$.
The extended Jacobi form
$\widetilde{\varphi}(Z)=\varphi(\tau, \ffZ)\exp(2\pi i \frac{\omega}{Q})$
of index $\frac{1}Q$
is holomorphic at ``infinity'' $({\Im \omega}\to +i\infty)$.
Therefore
$|\widetilde{\varphi}|$ is bounded  in any neighborhood of ``infinity''
(see \cite{CG} and \cite{Kl}).
We can rewrite this fact using the free parameter
$\widetilde{v}=\widetilde{v}(Z)>0$
from Lemma \ref{lem-v-free}. Then we have
$$
|\varphi(\tau, \ffZ)|
\exp\bigl (- \frac{2\pi }{Q}\, \frac{(\hbox{Im}(\ffZ),\hbox{Im}(\ffZ))}{2v}\bigr)<C
$$
is bounded for $v={\rm Im}(\tau)> \varepsilon$ and the exponential term does not depend
on the action of $\Gamma^J(L(Qt))$.
Using the action of $\SL_2(\ZZ)<\Gamma^J(L(Qt))$ we obtain that
$$
|\varphi(\tau, \ffZ)|
\exp\bigl(-\frac{2\pi}{Q}\frac{(\hbox{Im}(\ffZ),\hbox{Im}(\ffZ))}{2v}\bigr)
<Cv^{-k}
$$
if $v\leqslant \varepsilon$ (see \cite[\S 2]{CG} for similar considerations).
Now we can get an estimation of all terms in the sum for $\Lift_{\mu}(\varphi)$
for $v>\varepsilon$. We have
$$
|a^k\varphi(\frac{a\tau+bQ}{d}, a\ffZ)
\exp\bigl(-2\pi \frac{1}{Q}(\frac{(\hbox{Im}(a\ffZ),\hbox{Im}(a\ffZ))}{2va/d})\bigr)|
<C d^k v^{-k}
$$
if $\frac {a}{d}v\leqslant \varepsilon$. If $\frac {a}{d}v> \varepsilon$
then we have $< Ca^k$.
In the both cases we see that the term above depending on $(a,b,d)$
is smaller than $C_\varepsilon m^k$. It gives us
$$
|m^{-1}\varphi_m(\tau,\mathfrak{Z})e^{2i\pi\frac{m}{Q}\omega}|
<C_\varepsilon m^k\sigma_0(m)
\exp(-\frac{2\pi m}{Q}\widetilde{v})<
C_\varepsilon m^{k+1}\exp(-2\pi \frac{m}{Q}\widetilde{v})
$$
where $\widetilde{v}(Z)>0$. Therefore the function $\Lift_{\mu}(\varphi)$
is well defined and it transforms like a modular form of weight $k$
and  character $v_{\eta,\mu}^D\times \nu\times e^{2\pi i \frac{\mu}{Q} r }$
with respect to the parabolic  subgroup $\Gamma^J(L(Qt))$.
\smallskip

{\bf Fourier expansion} of $\Lift_{\mu}(\varphi)$.
In the summation of the Fourier expansion of
$\varphi\in
{J}_{k, L; t}(v_{\eta}^D \times \nu)$
we have $n\equiv \frac{D}{24}\bmod \ZZ$ (see \eqref{F-exp-phi}).
Rewriting $n$ in terms of the conductor $Q=\frac{24}{D}$, 
the Fourier expansion of the function $\varphi$ has the following form
$$
\varphi(\tau,\mathfrak{Z})=
\sum_{\substack{n\equiv 1\bmod Q,\,n\geqslant 0\\
\vspace{1\jot}
\ l\in \frac 1{2} L^\vee\\
\vspace{0.5\jot} 2nt-(l,l) \geqslant 0}}
f(nD,l)e^{2i\pi(\frac{n}{Q}\tau+(l,\mathfrak{Z}))}.
$$
After the summation over $b\bmod d$ in the action of the Hecke
operator we get
$$
m^{-1}(\widetilde{\varphi}\vert_k T_-^{(Q)}(m))\circ \pi_{Qt}(Z)=
$$
$$
\sum_{\substack{ad=m\\ \vspace{0.5\jot} a>0}}a^{k-1} v_{\eta}^D(\sigma_a)
\sum_{\substack{nd\equiv 1 \bmod Q, n\geqslant 0\\
\vspace{0.5\jot}
l\in \frac{1}{2}L^\vee \\
\vspace{0.5\jot} 2ndDt-(l,l)\geqslant 0}}
f(ndD,l)e^{2i\pi(\frac{na}{Q}\tau+a(l,\mathfrak{Z})+\frac{ad}{Q}\omega)}.
$$
So we have
\begin{align*}
\Lift_{\mu}(\varphi)(Z)=
\sum_{\substack{m\equiv \mu \bmod Q\\
\vspace{0.5\jot} m\geqslant 1}}\
&\sum_{\substack{ad=m\\ \vspace{0.5\jot} a>0}}
a^{k-1}
 v_{\eta}^D(\sigma_a)\\
&\sum_{\substack{nd\equiv 1 \bmod Q\\
\vspace{0.5\jot}
l\in\frac{1}{2} L^\vee\\\vspace{0.5\jot}
2{ndDt}-(l,l)\geqslant 0}}
f(ndD,l)e^{2i\pi(\frac{na}{Q}\tau+a(l,\mathfrak{Z})+\frac{ad}{Q}\omega)}.
\end{align*}
But $nd\equiv 1 \bmod Q \Leftrightarrow an\equiv \mu \bmod Q$
because for any $(\mu,24)=1$ we have $\mu^2\equiv 1\bmod 24$.
(We note that $24$ is the maximal natural number with this property).
Using this property we obtain the Fourier expansion of  the lifting
\begin{align*}
\Lift_{\mu}(\varphi)(Z)&=
\hspace{-3\jot}
\sum_{\substack{m,n\equiv \mu \bmod Q\\
\vspace{0.5\jot}
m,n\geqslant 1 \\ \vspace{0.5\jot}
l\in \frac{1}2L^\vee\\
\vspace{0.5\jot}
2nmDt-(l,l)\geqslant 0}}
\left(\sum_{\substack{a|(n,l,m)}} a^{k-1} v_{\eta}^D(\sigma_a)
f(\frac{nmD}{a^2},\frac{l}{a})\right)
e^{2i\pi(\frac{n}{Q}\tau+({al},\mathfrak{Z}))+\frac{m}{Q}\omega)}.
\end{align*}
We can reformulate the condition on the hyperbolic norm of the index
$(n,l,m)$ of the Fourier coefficient in the term of the lattice
$L(Qt)$: $2\frac{ndD}{Q}-\frac{1}{Qt}(l,l)\geqslant 0$.

The formula for the Fourier expansion is symmetric with respect to $\tau$
and $\omega$. The involution $V$ which permutes the isotropic vectors $e_1$ and $f_1$
in the second copy of the hyperbolic plane of the lattice
$U\oplus U_1\oplus L(-Qt)$ realizes  the transformation
$\tau\leftrightarrow\omega$ and  $\mathfrak{Z}\leftrightarrow\mathfrak{Z}$.
We see that $V\in \widetilde{\Orth}^+(2U\oplus L(-Qt))$, $\rm{det}(V)=-1$,
$J(V,Z)=1$ and
$$
\Lift_{\mu}(\varphi)\vert_k V=\Lift_{\mu}(\varphi).
$$
It is known (see \cite[p. 1194]{G-K3} or \cite[Proposition 3.4]{GHS2})
that
$$
\widetilde{\Orth}^+(2U\oplus L(-Qt))=\langle \Gamma^J(L(Qt)),\ V \rangle.
$$
Therefore $\Lift_{\mu}(\varphi)$ is a modular form of weight $k$
with a character of order $Q$ with respect to
$\widetilde{\Orth}^+(2U\oplus L(-Qt))$.
\end{proof}

\noindent
{\bf Remark to Theorem \ref{thm-lift}}.
If $\mu=1$ then $\Lift(\varphi)=\Lift_{1}(\varphi)\not\equiv 0$ because
its  first Fourier--Jacobi coefficient $\widetilde{\varphi}$
is not zero. For $\mu\ne 1$ the function $\Lift_{\mu}(\varphi)$ might be
identically zero.  See \cite[Example 1.15]{GN3} for
a non-zero $\mu$-lifting  in the case of signature $(2,3)$.
\smallskip

At the end of the section we give the first application of Theorem \ref{thm-lift}.

\noindent{\bf Example 2.3.} {\it Modular forms of singular weight}.
The first example of such modular forms
$$
\Lift(\Theta_{E_8})\in M_4(\Orth(II_{2,10}))
$$
was given  in \cite{G1}. This function is  sometimes called the simplest modular form (or the Gritsenko form) because it has very simple
Fourier coefficients.
Using the theta-products  \eqref{theta-Am}--\eqref{theta-Dm3}
we can define modular forms of singular weight on orthogonal groups
with a character induced by $v_\eta^D$-character for
$D=2$, $4$, $6$, $8$, $12$ and $24$.
We give some examples below in order to illustrate different cases:
\[
\Lift(\vartheta_{D_8})\in M_4(\widetilde{\Orth}^+(2U\oplus D_8(-1))),\quad
\Lift(\vartheta_{D_{8}(3)})\in M_{4}(\widetilde{\Orth}^+(2U\oplus D_{8}(-9)),\chi_3),
\]
\[
\Lift(\vartheta_{4A_1})\in M_2(\widetilde{\Orth}^+(2U\oplus 4A_1(-1)), \chi_2),\
\Lift(\vartheta_{D_{24}(3)})\in M_{12}(\widetilde{\Orth}^+(2U\oplus D_{24}(-3))),
\]
\[
\Lift(\vartheta_{2A_1})\in M_1(\widetilde{\Orth}^+(2U\oplus 2\latt{-4}), \chi_4),
\quad
\Lift(\vartheta_{D_{2}(3)})\in M_{1}(\widetilde{\Orth}^+(2U\oplus 2\latt{-36}),\chi_{12})
\]
where $\chi_n$ denotes a character of order $n$ of the corresponding orthogonal group.
\smallskip

We note that in many case {\it the maximal modular group} of the lifting
is larger than the stable orthogonal group
$\widetilde{\Orth}^+(2U\oplus L(-1))$. For example,
the maximal modular group of
$\Lift(\eta^d\vartheta_{D_{m}})$ for any $d$ and $m$ such that
$d+3m\equiv 0\bmod 24$
is the full orthogonal group ${\Orth}^+(2U\oplus D_m(-1))$
if $m\ne 4$. The form $\Lift(\eta^d\vartheta_{D_{m}})$
is anti-invariant with respect to the involution of the Dynkin diagram (the  reflection with respect to a vector with square $4$).
If $m=4$ then
$$
{\Orth}^+(2U\oplus D_4(-1))\,/\,\widetilde{\Orth}^+(2U\oplus D_4(-1))
\cong S_3.
$$
The liftings of $\vartheta_{D_4}$, $\vartheta^{(2)}_{D_4}$,
$\vartheta^{(3)}_{D_4}$ (see Example 1.8) are modular with respect
to three different subgroups of order $3$ in ${\Orth}^+(2U\oplus D_4(-1))$.
\smallskip

The lifting of any theta-products vanishes along the divisors
of the corresponding Jacobi forms. In particular
$\Lift(\vartheta_{4A_1})$
vanishes with order one  along $z_i=0$.
It is known that the full divisor of this modular form
is equal to the union of all modular transformations of $z_i=0$,
i.e. this is a singular reflective modular form
with the simplest possible divisor (see \cite{G10}).
The same is true for $\Lift(\vartheta_{D_8})$.
The Fourier expansion  of $\Lift(\vartheta_{4A_1})$
(or $\Lift(\vartheta_{D_8})$)  written in a fixed Weyl chamber
of the corresponding orthogonal group will define generators
and relations of Lorentzian Kac--Moody algebras
(see \cite{GN1}--\cite{GN3}
and a  forthcoming paper of Gritsenko and Nikulin
about reflective groups of rank $\geqslant4$).
Here we consider the formula for $4A_1$ which was given without
proof in \cite{G10}.
\smallskip

\noindent{\bf Example 2.4.} {\it Jacobi lifting,
the modular tower $4A_1$ and  modular
forms  of ``Calabi--Yau type''.}
We consider the following theta-product as a Jacobi form of index $\frac{1}2$
$$
\vartheta_{4A_1}(\tau, \mathfrak{Z}_4)
=\vartheta(\tau,z_1)\dots \vartheta(\tau,z_4)
\in J_{2,4A_1;\frac{1}2}(v_\eta^{12}\times v_H^{\otimes 4}).
$$
According to  Theorem \ref{thm-lift} we get
$$
\Phi_2(Z):=\Lift(\vartheta_{4A_1})(\tau, \mathfrak{Z}_4,\omega)
\in M_2(\Orth^+(2U\oplus 4A_1(-1)),\chi_2)
$$
where $\chi_2$ is a character of order $2$ of the full orthogonal
group. The modular form  $\Phi_2$ is reflective
with the simplest possible divisor (see  \cite{G10}).
The Fourier expansion of this fundamental reflective form of singular weight
is the following
$$
\Phi_2(Z)=\sum_{ \ell=(l_1,\dots, l_4), \ l_i\equiv \frac 1{2} \,{\rm mod \,}\ZZ}
$$
$$
\sum_{\substack{
 n,\,m\in \ZZ_{>0}\\
 \vspace{0.5\jot} n\equiv m\equiv 1\,{\rm mod\,}\ZZ\\
\vspace{0.5\jot}  nm-(\ell,\ell)=0}}
\sigma_1((n,\ell,m))
\biggl(\frac{-4}{2l_1}\biggr)\dots \biggl(\frac{-4}{2l_4}\biggr)
\,e^{\pi i (n\tau+ (2\ell,\mathfrak{Z}_4)+m\omega)}
$$
where $\sigma_1(n)=\sum_{d|n} d$.
The quasi-pullbacks (see \cite{GHS}) of $\Phi_2$
along the divisors is again  reflective (see \cite{G10}).
In this way we obtain the $4A_1$-tower
of reflective modular forms in six, five, four and three
variables with respect to $\Orth^+(2U\oplus nA_1(-1))$
for $n=4$, $3$, $2$ and $1$:
$$
\Phi_2=\Lift(\vartheta_{4A_1}),\qquad \Lift(\eta^3\vartheta_{3A_1}),
$$
$$
K_4(\tau,z_1,z_2,\omega):=\Lift(\eta^6(\tau)\vartheta(\tau,z_1)\vartheta(\tau,z_2)),
\quad
\Delta_5=\Lift(\eta^9(\tau)\vartheta(\tau,z))
$$
where $\Delta_5\in S_5(\Sp_2(\ZZ),\chi_2)$ is
the Igusa modular form (a square root of the first Siegel cusp form
of weight $10$). The modular form $\Delta_5$ determines one of the most
fundamental Lorentzian Kac--Moody algebras related
to the second quantized elliptic genus of $K3$ surfaces
(see \cite{GN1},  \cite{DMVV} and  \cite{G-EG}).
The modular form
$$
K_4
\in S_4\bigl(\widetilde{\SO}^+(2U\oplus 2A_1(-1)), \chi_2\bigr)
$$
is the second member of the modular $4A_1$-tower based on
$\Delta_5$. This form defines an (elliptic) Lorentzian
Kac--Moody algebra of signature $(1,3)$
(see a forthcoming paper of Gritsenko and Nikulin).
Moreover $K_4(Z)dZ$
is the only canonical differential form
on the orthogonal modular variety
$$
M_{\chi_2}(2A_1)=\Gamma_{\chi_2}\setminus {\cD}(2U\oplus 2A_1(-1))
$$
of complex dimension $4$ and  of Kodaira dimension $0$ where
$\Gamma_{\chi_2}=\ker(\chi_2)$ (see \cite{G10}).
The first example of cusp forms of this  type was considered
in \cite{GH1} where it was shown that the modular form
$$
\Delta_{1}=\Lift(\eta(\tau)\vartheta(\tau,z))
\in S_{1}(\widetilde{\Orth}^+(2U\oplus \latt{-6}), \chi_6)
$$
determines the unique, up to a constant, canonical differential form
$\Delta_1^3(Z)dZ$ on the Barth-Nieto modular Calabi--Yau three-fold.
The second example of Siegel cusp forms of canonical weight
with the simplest possible divisor was constructed
in \cite{CG}:
$$
\nabla_3=\Lift(\eta(\tau)\eta(2\tau)^4\vartheta(\tau,z))
\in S_3(\Gamma^{(2)}_0(2),\chi_2)
$$
where $\Gamma^{(2)}_0(2)<\Sp_2(\ZZ)$ and
$\chi_2$ is its  character of order $2$.
A Calabi-Yau model of the Siegel modular three-fold
$\Gamma^{(2)}_0(2)_{\chi_2}\setminus \HH_2$
was found in \cite{FS-M}.
The modular form in four variables $K_4$
is the next example of cusp form of ``Calabi--Yau type''
similar to the Siegel modular forms $\Delta_1^3$ and $\nabla_3$.
We can ask a question about the existence of a compact
model of Calabi--Yau type of the modular variety
$M_{\chi_2}(2A_1)$ of dimension $4$
defined above.

\section {Modular forms of singular and  critical weights}

The minimal possible weight (singular weight) of holomorphic Jacobi form
for $L$ is $\frac{n_0}2$ where $n_0=\rank L$.
The first weight for which Jacobi cusp forms might appear is equal to $\frac{n_0+1}2$.
This weight is called {\it critical}.
In the case of classical modular forms
in one variable the critical weight is equal to $1$.
The simplest possible example of  modular forms
of critical weight in our context is the cusp form
$\Delta_{1}=\Lift(\eta\,\vartheta)$ of weight $1$
with a character of order $6$
for the lattice $2U\oplus \latt{-6}$ of signature $(2,3)$.
We mentioned in Example 2.4 above that
this function determines one of the basic Lorentzian Kac--Moody algebras
in the Gritsenko--Nikulin classification
(see \cite{GN2}--\cite{GN3}) and it induces the unique canonical differential
form on a special Calabi--Yau three-folds, the Barth--Nieto quintic.
We can construct a simple example of
modular form of critical weight with trivial character
using Theorem \ref{thm-lift}. This is
$$
\Lift(\eta\,\vartheta_{D_{23}(3)})
\in M_{12}(\widetilde{\Orth}^+(2U\oplus D_{23}(-3)))
$$
which is a modular form with trivial character
with respect to  the orthogonal group of signature $(2,25)$.
In this section we construct examples of Jacobi cusp forms of critical
weight for  all even ranks.
For this aim we use the pullback of Jacobi forms of singular weight
such that its Fourier coefficient $f(0,0)=0$.
This is exactly the case of $\vartheta_{D_{m}}$.

Let $M<L$ be an even  sublattice of $L$.
We can consider the Heisenberg group of $M$
as a subgroup of $H(L)$. Therefore if ${\rm rank} (M)={\rm rank} (L)$
then the Jacobi forms with respect to $L$ can be considered as Jacobi forms with respect to $M$.
In the next proposition we consider the operation of pullback.
\begin{proposition}\label{prop-pback}
Let $M<L$ be a sublattice of $L$ and ${\rm rank} (M)<{\rm rank} (L)$
$$
M\oplus M^\perp < L,\qquad
\ffZ=\ffZ_m\oplus \ffZ_{\perp}\in L\otimes \CC= (M\oplus M^\perp)\otimes \CC.
$$
For any $\varphi(\tau,\ffZ)\in J_{k,L;t}(\chi\times \nu)$
its pullback is also a Jacobi form
$$
\varphi|_{M}:=\phi(\tau,\ffZ_m)=\varphi(\tau,\ffZ)|_{(\ffZ_\perp=0)}
\in J_{k,M;t}(\chi\times \nu|_{\Gamma^J(M)}).
$$
The pullback of  a  Jacobi cusp form  is a cusp form or $0$.
\end{proposition}
\begin{proof}
We note that the pullback of Jacobi form might be the zero function.
What is more interesting is that the pullback might be  a cusp form
although the original function is not.

The functional equations \eqref{jacobi-A}--\eqref{jacobi-H}
are  evidently true for $\varphi|_M$.
To calculate its Fourier expansion
we consider the embedding of the lattices
$$
M\oplus M^\perp < L<L^\vee < M^\vee \oplus (M^\perp)^\vee.
$$
We have to analyze the $M$-projection of any  vector
$l$ in $\frac{1}2L^\vee=L(2)^\vee$ in the Fourier expansion
\eqref{F-exp-phi}.
If the character $\nu$ of the minimal Heisenberg group  is trivial
then we do not need the coefficient $\frac{1}2$ before the lattices
dual to $L$ and $M$ in the calculation below.
For any $l\in \frac{1}2L^\vee=L(2)^\vee$ we have
the following  decomposition
$$
l=l_m\oplus l_{\perp}={\rm pr}_{M(2)^\vee} (l)\oplus {\rm pr}_{(M(2)^\perp)^\vee}(l)
\in M(2)^\vee \oplus (M(2)^\perp)^\vee.
$$
In the coordinates  $\ffZ=\ffZ_m\oplus \ffZ_{\perp}$ we have
$$
\varphi(\tau, \ffZ)=\sum_{n\geqslant0,\,l=l_m\oplus\, l_{\perp}}
f(n,l)e^{2\pi i({n}\tau+(l_m,\ffZ_m)+(l_\perp,\ffZ_\perp))}.
$$
Therefore
$$
\varphi|_{M}(\tau, \ffZ_m)=\sum_{n\geqslant0,\,l_m\in M(2)^\vee}
\bigl( \sum_{\substack{l_\perp\in (M(2)^\perp)^\vee\\
l_m\oplus\, l_{\perp}\in L(2)^\vee}}
f(n, l_m\oplus\, l_{\perp})\bigr)\,  e^{2\pi i({n}\tau+(l_m,\ffZ_m))}.
$$
We note that $2nt-(l_m,l_m)\geqslant(l_\perp, l_\perp)\geqslant0$.
The last inequality is strict
if $\varphi$ is a cusp form.
\end{proof}
Using the operation of pullback we can construct
Jacobi {\it cusp} forms of critical weight starting
from Jacobi forms of singular weight if  the constant term
$f(0,0)$ of the last one is equal to zero.
The estimation on $2nt-(l_m,l_m)$ at the end of the proof
of the last proposition
gives us the following estimation of the order at infinity
(see \eqref{ord-phi}) of the pullback.
\begin{corollary}\label{col-ordM}
In the conditions of Proposition \ref{prop-pback} we have
$$
{\rm Ord}(\varphi|_{M}(\tau, \ffZ_m))\geqslant \min \{(l_\perp, l_\perp)\,|\  l_\perp=
{\rm pr}_{(M^\perp)^\vee}(l)\ {\rm such\  that\  } f(n,l)\ne 0\}.
$$
In particular, if ${\rm pr}_{(M^\perp)^\vee}(l)\ne 0$ for all $f(n,l)\ne 0$
then  the pullback $\varphi|_{M}$ is a cusp form or the zero function.
\end{corollary}

Using the last corollary  we can construct new important examples
of Jacobi forms of singular and critical weights.
We recall that by $J_{k,L}$ we denote the space of Jacobi forms
of index one.

We define the root lattice $A_m$ as a sublattice of $D_{m+1}$
$$
A_m=\{(x_1,\dots,x_{m+1})\in \ZZ^{m+1}\,|\, x_1+\dots+x_{m+1}=0\} < D_{m+1}.
$$
We  note that
$A_1\cong \latt{2}$, $A_1\oplus A_1\cong D_2$ and  $A_3\cong  D_3$.

\begin{proposition}\label{Ptheta-Am}
{\rm 1)} Let $v=2(b_1,\ldots ,b_m)\in  \ZZ^m$ be an element in  $D_m$
with at least two non-zero coordinates $b_i$  such that
$(b_1+\ldots +b_m)\equiv 1 \mod 2$ and g.c.d.$(b_1,\dots, b_m)=1$.
Then
$$
\vartheta_{D_m}|_{v^\perp}
\in J_{\frac{m}2, v^\perp_{D_m}}^{cusp} (v_\eta^{3m})
\quad{and}\quad
\vartheta_{D_m(3)}|_{v^\perp}
\in J_{\frac{m}2, v^\perp_{D_m(3)}}^{cusp} (v_\eta^{m})
$$
is a non-zero Jacobi cusp form of critical weight such that
$$
{\rm Ord}(\vartheta_{D_m}|_{v^\perp})= \frac{1}{(v,v)}>0
\qquad
{\rm and}\qquad
{\rm Ord}(\vartheta_{D_m(3)}|_{v^\perp})= \frac{1}{3(v,v)}>0.
$$
\newline
{\rm 2)}  The theta-product
$$
\vartheta_{A_m}(\tau, z_1,\dots,z_m)
=\vartheta(\tau, z_1)\cdot \ldots\cdot \vartheta(\tau, z_m)
\cdot\vartheta(\tau, z_1+\ldots+z_m)\in J_{\frac{m+1}2, A_m}(v_\eta^{3m+3})
$$
is a Jacobi  form of critical weight.
If $m$ is  even then $\vartheta_{A_m}$ is a Jacobi cusp form
and
$$
{\rm Ord}(\vartheta_{A_m})=\frac{1}{4(m+1)}>0.
$$
{\rm 3)} For the renormalized lattice $A_m(3)$ the Jacobi form
$$
\vartheta_{A_m(3)}(\tau, z_1,\dots,z_m)
=\vartheta_{3/2}(\tau, z_1)\cdot \ldots\cdot \vartheta_{3/2}(\tau, z_m)
\cdot \vartheta_{3/2}(\tau, z_1+\ldots+z_m)
$$
belongs to  $J_{\frac{m+1}2, A_m(3)}(v_\eta^{m+1})$.
For even $m$
$$
{\rm Ord}(\vartheta_{A_m(3)})=\frac{1}{12(m+1)}>0.
$$
\end{proposition}
\begin{proof}
1) If in $v$ only one $b_i\ne 0$ then $\vartheta_{D_m}|_{v^\perp}\equiv 0$.
To prove the lemma we calculate the Fourier expansion
of the pullback function.
The discriminant group of $D_m$ was given in Example 1.9.
The Fourier expansion  of $\vartheta_{D_m}$  has the following form
$$
\vartheta_{D_m}(\tau, \ffZ_m)=
\sum_{\substack{  n\in \QQ_{>0},\ l \in \frac{1}2 \ZZ^m
\vspace{0.5\jot} \\
 2n-(l,l)= 0}}
f(n,l)\, e^{2\pi i (n\tau+(l,\ffZ_m))}
$$
where
$$
f(n,l)=\biggl(\frac{-4}{2l}\biggr)=
\biggl(\frac{-4}{2l_1}\biggr)\cdot \ldots
\cdot\biggl(\frac{-4}{2l_m}\biggr)
$$
is the product of the generalized Kronecker symbols modulo $4$.
In particular all coordinates $2l_i$ are odd.
If $v$ is a vector which satisfies the condition of the proposition
then $(l, v)=2(l_1b_1+\ldots+l_mb_m)\equiv 1 \mod 2$.
The  lattice $\latt {v}^\vee$ is
generated by $\frac{v}{(v,v)}$. We get that
$$
l_\perp={\rm pr}_{\latt {v}^\vee}(l)=(l,v)\frac v{(v,v)} \ne 0
$$
is always non trivial.
Moreover, there exists a vector $2l=(2l_i)$ with  odd coordinates
such that $(l,v)=1$.
According to Corollary \ref{col-ordM}
$$
{\rm Ord}(\vartheta_{D_m}|_{v^\perp})=|(\frac{v}{(v,v)}, \frac{v}{(v,v)})|=\frac{1}{|(v,v)|}>0
$$
and  $\vartheta_{D_m}|_{v^\perp}$ is  a Jacobi cusp form.
The proof for $D_m(3)$ is quite similar.

2) We have $A_m=v^{\perp}_{D_{m+1}}$ where $v=2(1,\dots,1)\in D_{m+1}$.
In particular $z_{m+1}=-(z_1+\ldots +z_{m})$ and
$\vartheta_{A_m}=-\vartheta_{D_{m+1}}|_{v^\perp}$. If $m$ is even then
$v$ satisfies the condition in 1) and
${\rm Ord}(\vartheta_{A_m})=\frac{1}{4(m+1)}$.
The proof of 3) is similar.
\end{proof}

\noindent
{\bf Example.} Since $A_3\cong D_3$ there exist a Jacobi form of singular
and two Jacobi forms (cusp and non-cusp) of critical weight
for this lattice
$$
\vartheta_{D_3}\in J_{\frac{3}2,D_3}(v_\eta^9),\quad
\eta\,\vartheta_{D_3}\in J^{cusp}_{2,D_3}(v_\eta^{10}),\quad
\vartheta_{A_3}\in J_{2,A_3}(v_\eta^{12}).
$$

We can construct many   Jacobi forms of singular, critical
and other small weights using
the equalities of the previous proposition.
For any $\varphi\in J_{k, L;t}(\chi)$ we denote by $\varphi^{[n]}$
the direct (tensor) product of $n$-copies
of $\varphi$, i.e. the Jacobi form for the lattice $nL$
$$
\varphi^{[n]}(\tau, (\ffZ_1,\ldots, \ffZ_n))=
\varphi(\tau, \ffZ_1)\cdot\ldots\cdot \varphi(\tau, \ffZ_n)\in J_{nk, nL;t}(\chi^n).
$$
The next example is very important.
\begin{corollary}\label{singJ-A2}
There exists a  Jacobi form of singular weight for $A_2$
\begin{equation}\label{sigmaA2}
\sigma_{A_2}(\tau, z_1,z_2)
=\frac{\vartheta(\tau,z_1)\vartheta(\tau,z_2)\vartheta(\tau,z_1+z_2)}
{\eta(\tau)} \in J_{1,A_2}(v_\eta^8).
\end{equation}
In particular the Jacobi form of singular weight
$\sigma_{3A_2}=\sigma_{A_2}^{[3]}\in  J_{3,3A_2}$
has trivial character.
\end{corollary}
\begin{proof}
We  note that
$\eta(\tau)^{-1}=q^{-1/24}(1+q(\dots))$. Therefore
$$
{\rm Ord}(\frac{\varphi(\tau, \ffZ)}{\eta(\tau)})=
{\rm Ord}(\varphi(\tau, \ffZ))-\frac{1}{12}.
$$
Therefore $\sigma_{A_2}$ is holomorphic Jacobi form of singular weight
for $A_2$.
\end{proof}

{\bf Remarks}. 1)
The Jacobi form $\sigma_{A_2}$ is equal to the denominator
function of the affine Lie algebra $A_2$ (see \cite{KP} and \cite{D}).
We consider the Jacobi forms related to the  denominator functions
of all affine Kac--Moody Lie algebras
in a forthcoming paper of V. Gritsenko and  K.-I. Iohara.

2) The lifting of
$\sigma_{3A_2}$ is a reflective modular form of singular weight.
The lifting of $\eta^8\sigma_{2A_2}$ determined
the unique canonical differential form on a modular variety
of Kodaira dimension $0$  (see \cite{G10}).

3) The form $\sigma_{A_2}$ is the first example
of Jacobi form obtained as  theta/eta-quotients.
Using such Jacobi form we can  produce important classical Jacobi forms
in one variable called {\it theta-blocks}.
See Corollary  \ref{quarks} and \cite{GSZ}.
\smallskip

Using the same principle we obtain
\begin{corollary}\label{critJ-2A4}
The Jacobi forms given below are cusp forms of critical weight.
$$
\kappa_{2A_4}=\frac{\vartheta_{A_4} \otimes \vartheta_{A_4}}{\eta}
\in J^{cusp}_{\frac{9}2, 2A_4}(v_\eta^5),
\qquad
\kappa_{A_4\oplus A_6}=
\frac{\vartheta_{A_4}\otimes \vartheta_{A_6}}{\eta}
\in J^{cusp}_{\frac{11}2, A_4\oplus A_6}(v_\eta^{11}).
$$
Let
$v_5=2(2,1,0,\dots,0)\in D_m$ ($m\geqslant2$)
and $v_7=2(2,1,1,1,0,\dots,0)\in D_n$ ($n\geqslant4$).
Then
$$
\frac{\vartheta_{D_m}|_{v_5^\perp}\otimes\vartheta_{D_n}|_{v_a^\perp}}
{\eta}
\in J^{cusp}_{\frac{m+n-1}{2}, {D_m}|_{v_5^\perp}\otimes
{D_n}|_{v_a^\perp}}(v_{\eta}^{3(n+m)-1})
$$
where $a=5$ or $7$.
\end{corollary}
\begin{proof} According to Proposition \ref{Ptheta-Am}
$$
{\rm Ord}(\kappa_{2A_4})=\frac{1}{60},\
{\rm Ord}(\kappa_{A_4 \oplus A_6})=\frac{1}{420},\
{\rm Ord}(\vartheta_{D_m}|_{v_5^\perp})=\frac{1}{20},\
{\rm Ord}(\vartheta_{D_m}|_{v_7^\perp})=\frac{1}{28}.
$$
\end{proof}
\noindent
{\bf Remark.}
In the same way  we get {\it non-cusp} Jacobi forms of weight
$\frac{n_0}2+1$ (singular weight $+1$):
$$
\frac{\vartheta_{A_4}^{[5]}}{\eta^3}
\in J_{11, 5A_4},\quad
\frac{\vartheta_{A_6}^{[7]}}{\eta^3}
\in J_{23, 7A_6},\quad
\frac{\vartheta_{A_8}^{[3]}}{\eta}
\in J_{13, 3A_8}(v_\eta^8).
$$
The first two functions have trivial character. It might be that these functions are interesting Eisenstein series. We can also mention
the non-cusp form
$$
\vartheta_{A_2(3)}^{[3]}/\eta\in
J_{4,3A_2(3)}(v_\eta^8).
$$

\begin{proposition}\label{Form-sing}
Theta-products give  examples of Jacobi forms
of singular weight with trivial character for
some lattices
of all even ranks $\geqslant 6$.
\end{proposition}
\begin{proof}
The corresponding Jacobi form are tensor products
of the following Jacobi theta-products
\begin{align*}
\sigma_{A_2}
&\quad {\rm with\  character\ } v_\eta^8,\qquad\qquad\qquad
\vartheta_{D_m}
\quad {\rm with\  character\ }v_\eta^{3m},\\
\vartheta^{(i)}_{D_4}&
\quad {\rm with\  character}\ v_\eta^{3m}\ (i=2,3),\quad
\vartheta_{2A_1}^{(1)}
\quad {\rm with\  character}\ v_\eta^{6},\\
&\qquad\qquad\qquad\qquad\vartheta_{D_m(3)}
\quad {\rm with\  character\ }  v_\eta^{m}.
\end{align*}
See \eqref{sigmaA2}, \eqref{theta-Dm}, \eqref{theta-D4a}, \eqref{theta-D4b},
\eqref{theta-D2} and \eqref{theta-Dm3}.
Below we give a list of lattices of rank smaller
or equal to $24$ since for larger ranks one can use the periodicity
of the characters:
$$
n=6,\ 3A_2;\quad
n=8,\  D_8,\ 2D_4,\   8A_1; \quad
n=10,\ D_7\oplus D_3(3),\ A_2\oplus D_4\oplus D_4(3);
$$
$$
n=12,\ D_6\oplus D_6(3),\ 2A_2\oplus D_8(3),\ 6A_2;\quad
n=14,\ D_5\oplus D_9(3);\quad
$$
$$
n=16,\ D_{16},\ D_4\oplus D_{12}(3);\quad n=18,\ D_3\oplus D_{15}(3);\qquad
n=20,\ D_2\oplus D_{20}(3);
$$
$$
n=22,\ D_1\oplus D_{21}(3);\quad
n=24,\ D_{24},\  12A_2,\ D_{24}(3).
$$
We note that we consider  $8A_1$ as $4(A_1\oplus A_1)$.
The corresponding Jacobi form is the product of four functions
of type $\vartheta_{2A_1}^{(1)}$.
Moreover instead of any $D_m$ in the list above we can put
a direct sum
$D_{m_1}\oplus\dots \oplus D_{m_k}$ with $m_1+\dots+m_k=m$.
\end{proof}

Jacobi forms of singular weight with respect to the full Jacobi
group of a lattice $L$ have a $\SL_2(\ZZ)$-character of type
$v_\eta^{2m}$ if the rank of $L$ is even (the singular weight
is integral) or a multiplier system of type $v_\eta^{2m+1}$
if the rank is odd  (the singular weight  is half-integral).
Analyzing the examples of theta-products given above
we get the following table of possible characters $v_\eta^m$
\begin{align*}
{\rm \bf rank}\ n&\quad d: {\rm \bf character\ of\  type\ } v_\eta^d\\
1&\quad 1,3\\
2&\quad 2,4,6,8\\
3&\quad 3,5,7,9,11\\
4&\quad 4,6,8,10,12,14,16\\
5&\quad 5,7,9,11,13,15,17,19\\
6&\quad 6,8,10,12,14,16,18,20,22,24\\
7&\quad 7,9,11,13,15,17,19,21,23,1,3\\
8&\quad 8,10,12,14,16,18,20,22,24,2,4,6.
\end{align*}
As a corollary we obtain
\begin{proposition}\label{sing-char}
If $n\geqslant8$ is even (respectively, $n\geqslant9$ is odd)
and $d\equiv n\bmod 2$
then there exists a lattice $L$ of rank $n$
such that the space of Jacobi forms of singular weight
$J_{\frac{n}2,L}(v_\eta^{d})$
is not empty.
\end{proposition}
\noindent
{\bf Remark.} For some $n$ we can prove that the table above
contains all possible characters. We are planning to come to this question
in another publication.
\smallskip

Now we would like to analyze Jacobi forms of critical weight.
First we note that the multiplication by $\eta$ gives us
the simplest such Jacobi form
\begin{equation}\label{J-D23}
\eta\,\vartheta_{D_{23}(3)}
\in J_{12,D_{23}(3)}.
\end{equation}
The tensor product of a Jacobi form of singular weight
and Jacobi form of critical weight has critical weight
for the corresponding lattice.
In particular there exist two simple series of Jacobi cusp forms
with trivial character for the lattices $A_m\oplus D_n$ where $m$ is even
and $m+n\equiv 7\mod 8$
\begin{equation}\label{J-AmDn}
\vartheta_{A_m\oplus D_n}=
\vartheta_{A_m}(\tau,\ffZ_m)\otimes \vartheta_{D_n}(\tau,\ffZ_n)
\in J^{cusp}_{(m+n+1)/2,\, A_m\oplus D_n},
\end{equation}
\begin{equation}\label{J-AmD3n}
\vartheta_{A_m\oplus D_{3n}(3)}=
\vartheta_{A_m}(\tau,\ffZ_m)\otimes \vartheta_{D_{3n}(3)}(\tau,\ffZ_{3n})
\in J^{cusp}_{(m+3n+1)/2,\, A_m\oplus D_{3n}(3)}.
\end{equation}
In particular we get  examples of Jacobi cusp forms
of weight one with character in one abelian variable.
The simplest examples of such forms can be found in
\cite{GN2} (see also \cite{GH2} where many different cusp  theta-products
of small weights were considered):
$$
\eta(\tau)\vartheta_{3/2}(\tau,2z)\in J_{1,D_1(3)}(v^2_\eta),
\qquad
\eta(\tau)\vartheta(\tau,2z)\in J_{1,D_1}(v^4_\eta).
$$
To get more interesting examples we take the pullback
of $\sigma_{A_2}$.
We consider $A_2$ as the sublattice $v^\perp_{D_3}$
where $v=2(1,1,1)$. Let $u=2(u_1,u_2,u_3)\in D_3$.
Let $u_a$ be the projection of $u$ on $A_2^\vee$,
i.e. $u=u_a+u_v$ where $u_a\in \latt{v^\perp}^\vee=A_2^\vee$
and $u_v\in \latt{v^\vee}=\latt{\frac{v}{12}}$.
We put
$\sigma_{A_2}|_u:=\sigma_{A_2}|_{(u_a)^\perp_{A_2}}$.

\begin{proposition}\label{Pb-A2}
Let $u=2(u_1,u_2,u_3)\in D_3$ such that
$u_i\ne u_j$ and $u_1+u_2+u_3\not\equiv 0\bmod 3$.
Then $\sigma_{A_2}|_u$ is a Jacobi cusp form of critical
weight $1$ with character $v_\eta^8$.
\end{proposition}
\begin{proof}
We note first that if $u_i\ne u_j$ then the pullback
$\sigma_{A_2}|_u$ is not equal to  zero identically.
According to the proof of Proposition \ref{prop-pback}
and Corollary \ref{singJ-A2} the Fourier expansion
of Jacobi  form $\sigma_{A_2}$ of singular weight $1$
has the following form
$$
\sigma_{A_2}(\tau, \ffZ_2)=
\sum_{\substack{n>0,\, l_a\in A_2^\vee
\vspace{0.5\jot}\\
2n-(l_a,l_a)=0\vspace{0.5\jot} \\
l_a\pm v^\vee\in \frac{1}2\ZZ^3}}
f(n, l_a)\,  e^{2\pi i({n}\tau+(l_a,\ffZ_2))}.
$$
More exactly, in the last summation we have
$l^\vee=l_a\pm \frac {v}{12}=\frac{1}{2}(l_1,l_2,l_3)$
with odd $l_i$ because the division by $\eta$
does not change the $\ffZ_2$-part of $\vartheta_{A_2}$.
Let $u_a$ be the projection of $u$ on $A_2$,
i.e. $u=u_a+u_v$ where $u_a\in \latt{v^\perp}^\vee=A_2^\vee$
and $u_v\in \latt{v^\vee}=\latt{\frac{v}{12}}$.
We have to analyze the Fourier expansion of $\sigma_{A_2}|_u$.
As  in the proof of Proposition \ref{prop-pback}
we put $l_a=l_u\oplus l_\perp$ where $(l_u, u_a)=0$.
If the hyperbolic norm of the index of a Fourier coefficient
$f_u(n,l_u)$ of $\sigma_{A_2}|_{u}$
is equal to zero then $l_\perp=0$. Therefore
$(l_u,u_a)=(l_a,u_a)=(l_a,u)=0$ and
$$
(l^\vee,u)=\pm \frac{(u,v)}{12}=\pm \frac{u_1+u_2+u_3}{3}\in \ZZ.
$$
The last inclusion is not possible. Thus
the pullback $\sigma_{A_2}|_u$ is a cusp form.
\end{proof}
We note that the Jacobi form in Proposition 3.8 is a
classical Jacobi form of type \cite{EZ}.
We give its more explicit form in the next
\begin{corollary}\label{quarks}
Let $a,b\in \ZZ_{>0}$. The following function,
called theta-quark,
$$
\theta_{a,b}(\tau,z)=
\frac{\vartheta(\tau, az)\vartheta(\tau, bz)\vartheta(\tau, (a+b)z)}
{\eta(\tau)}
\in J_{1,A_1;a^2+ab+b^2}(v_\eta^8)
$$
is holomorphic Jacobi form of Eichler--Zagier type of weight $1$,
index $(a^2+ab+b^2)$  and character $v_\eta^8$.
This is a Jacobi cusp form if $a\not \equiv b\mod 3$.
\end{corollary}
\begin{proof}We can assume that $a$ and $b$ are coprime.
We obtain this function as $\sigma_{A_2}|_u$  for $u=2(b,-a,0)$.
\end{proof}
\noindent
{\bf Remark.} The Jacobi form $\theta_{a,b}$ was proposed
by the second author many years ago in his talks on canonical
differential forms on Siegel modular three-folds.
The Jacobi forms of similar types, called {\it theta-blocks},
are studied in the paper \cite{GSZ} where
the Fourier expansion of theta-quark $\theta_{a,b}$
is found explicitly.
The method of the proof of Proposition \ref{Pb-A2}
can be used for other Jacobi forms when one takes a pullback
on a sublattice of co-rank  $2$.

Propositions \ref{Ptheta-Am} and  \ref{Pb-A2} give
a method to pass from Jacobi forms of singular weight
to Jacobi forms of critical weight. We have noticed that
the tensor product of Jacobi forms of singular and critical
weights is a form of critical weight. In some cases we can divide
some  products of two forms of critical weight by $\eta$
(see Corollary \ref{critJ-2A4}).
We  can control that the obtained Jacobi form is a cusp (or non-cusp)
form. Analyzing  the table of characters before
Proposition \ref{sing-char} we obtain
\begin{proposition}\label{crit-char}
If $n\geqslant7$ is odd (respectively, $n\geqslant8$ is even)
and $d\equiv n+1\bmod 2$
then there exists a lattice $L$ of rank $n$
such that the space of Jacobi forms of critical  weight
$J^{(cusp)}_{\frac{n+1}2,L}(v_\eta^{d})$ is not empty.
\end{proposition}

The analogue of Proposition \ref{Form-sing} is the following
\begin{proposition}\label{Form-crit}
Theta-products give  examples of Jacobi cusp  forms
of critical  weight with trivial character for
some lattices of all odd ranks $\geqslant5$.
\end{proposition}
\begin{proof}
The corresponding Jacobi forms of critical weight are
pullbacks (see Proposition \ref{Ptheta-Am})
of Jacobi forms of singular weight of Proposition \ref{Form-sing}.
One can also use  $\vartheta_{A_m(3)}$ instead of
$\vartheta_{A_m}$ in theta-products.
\end{proof}

\bibliographystyle{plain}

\medskip
\noindent
F.~Cl\'ery\\
Korteweg de Vries Instituut voor Wiskunde\\
Universiteit van Amsterdam\\
P.O. Box 94248\\
1090 GE AMSTERDAM\\
f.l.d.clery@uva.nl
\bigskip

\noindent
V.~Gritsenko\\
University Lille 1\\
Laboratoire Paul Painlev\'e\\
F-59655 Villeneuve d'Ascq, Cedex\\
Valery.Gritsenko@math.univ-lille1.fr\\

\end{document}